\documentclass[letterpaper, 10pt, conference]{ieeeconf}

\usepackage{amssymb}
\usepackage{amsmath,amsfonts,mathrsfs}
\newtheorem{Lemma}{Lemma}
\newtheorem{Corollary}{Corollary}

\newtheorem{Definition}{Definition}
\newtheorem{Remark}{Remark}
\newtheorem{Notation}{Notation}
\newtheorem{Theorem}{Theorem}

\newtheorem{Proposition}{Proposition}
\newtheorem{Problem}{Problem}
\newcommand{\Rank}{\mathrm{Rank}}
\newcommand{\LIN}{\mathcal{S}}
\newcommand{\LLIN}{\mathcal{L}}
\newcommand{\V}{\mathcal{V}}
\newcommand{\IM}{\mathrm{im}}
\newcommand{\LMAP}{\mathcal{M}}
\newcommand{\LOPT}{\mathcal{M}_{opt}}

\begin{document}
\title{Infinite horizon control and minimax observer design for linear DAEs}
\date{}
\author{Sergiy Zhuk$^\dag$ and Mihaly Petreczky$^\ddag$ \\
      $^\dag$ IBM Research, Dublin, Ireland, \texttt{sergiy.zhuk@ie.ibm.com},\\
      $^\ddag$Ecole des Mines de Douai, Douai, France, \texttt{mihaly.petreczky@mines-douai.fr}}

\maketitle
\begin{abstract}
In this paper we construct an infinite horizon minimax state observer for a linear stationary
differential-algebraic equation (DAE) with uncertain but bounded input and noisy output.
We do not assume regularity or existence of a (unique) solution for any initial state
of the DAE.
Our approach is based on a generalization of Kalman's duality principle.
%The latter allows us to transform minimax state estimation problem into a dual control problem for the adjoint DAE:
%the state estimate in the original problem becomes the
%control input for the dual problem and
%the cost function of the latter is, in fact, the worst-case estimation error.
%Using geometric control, we construct an optimal control in the feed-back form and represent it as an output of a stable LTI system. The latter gives the minimax state estimator in a recursive form.
In addition, we obtain a solution of  infinite-horizon linear quadratic optimal control problem for DAE.
\end{abstract}

\section{Introduction}
\label{sec:introduction}
%In this paper we present a solution of the following minimax state estimation problem:
 Consider a linear Differential-Algebraic Equation (DAE) with state $x$, output $y$ and noises $f$ and $\eta$:
\begin{align*}
  &\dfrac {d(Fx)}{dt}=Ax(t)+f(t), \quad Fx(t_0)=x_0\,, \\ %\label{eq:state_introduction}\\
  &y(t) = Hx(t) +\eta(t) %\label{eq:observation_introduction}\,,
\end{align*}
where $F, A\in\mathbb R^{n \times n}$, $H\in\mathbb R^{p\times n}$.
We do not restrict DAE's coefficients, in particular, we do not require that it has a
solution for any initial condition $x_0$ or that this solution is unique.
The only assumption we impose is that $x_0$, $f$ and $\eta$ are uncertain but bounded and belong to an ellipsoid in $L^2$. We will consider only solutions which are
locally integrable functions. We would like to estimate a state component $\ell^TFx(t)$, $\ell \in \mathbb{R}^n$ of the DAE based on the output $y$. The desired \emph{observer} should be linear in $y$, i.e. we are looking for maps $U(t,\cdot)\in L^2$ such that the estimate of $\ell^TFx(t)$ at time $t$ is of the form $\int_0^{t} U(t,s)y(s)ds$.
The goal of the paper is to find an observer $U$ such that:
\begin{enumerate}
\item
  The worst-case asymptotic observation error $\limsup_{t \rightarrow \infty} \sup_{f,\eta} (\ell^T Fx(t)-\int_0^{t} U(t,s)y(s)ds )^2$ is minimal, and
\item
  $U$ can be implemented by a stable LTI system, i.e.
  the estimate $t \mapsto \int_0^t U(t,s)y(s)ds$
  should be the output of a stable LTI system whose input is $y$.
\end{enumerate}
 We will call the observers defined above \emph{minimax observers}.

\textbf{Motivation}
The minimax approach is one of many classical ways to pose a state estimation problem.
We refer the reader to \cite{Tempo1985,Chernousko1994,Nakonechnii1978} and
\cite{Kurzhanski1997} for the basic information on the minimax framework.
Apart from pure theoretical reasons our interest in the minimax problem
is motivated by applications of DAE state estimators in practice.
In~\cite{Zhuk2012sysid} we briefly discussed one application of DAEs to non-linear filtering problems.
Namely, it is well known (see~\cite{GihmanSkorokhod1997}) that the density of a wide class of
non-linear diffusion processes solves forward Kolmogorov equation. The latter is a linear
parabolic PDE and its analytical solution is usually unavailable.
Different approximation techniques exist, though. One can project the density onto a finite
dimensional subspace and derive a DAE for the projection coefficients.
The resulting DAE will contain additive noise terms which represent the
projection error (see~\cite{MalletZhukPhyscon2011,ZhukDAEGalerkin2013} for details).
%Since the resulting DAE has uncertain initial condition (projection of the initial density)
%and model error (dynamic projection error), it seems natural to apply the minimax method in order to estimate the DAE's state, so to construct a worst-case estimate of the projection
%coefficients.
The worst-case state estimates of this DAE can be used to construct a state estimate of the non-linear diffusion process.

Besides, DAEs have a wide range of applications, without claiming completeness, we mention
robotics~\cite{Schiehlen2000}, cyber-security~\cite{PasqualettiCyberSecurity}
and modeling various systems \cite{muller2000proscons}.
We conjecture that the results of this
paper will be useful for many of the domains in which DAEs are used.

\textbf{Contribution of the paper}
 In this paper we follow the procedure proposed in~\cite{Zhuk2012sysid}: first,
 we apply a generalization of
 Kalman's duality principle in order to transform the minimax estimation
 problem  into a dual optimal control problem for the adjoint DAE.
 The latter control problem is an infinite horizon linear quadratic optimal control problem
 for DAEs.
  Duality allows us to view the observer $U$ as a control input of the adjoint system and
  to view the worst-case estimation error
  $\limsup_{t \rightarrow \infty} \sup_{f,\eta} (\ell^TFx(t)-\mathcal{O}_{U}(t))^2$ as the
  quadratic cost function of the dual control problem.
  Thus, the solution of the dual control yields an observer whose worst-case asymptotic
  error is the minimal one.
  The resulting dual control problem is then solved by translating it to a
  classical optimal control problem for LTIs. The solution of the latter
  problem yields a stable autonomous LTI systems, whose output is the solution of dual control problem.
  %The latter LTI can then be used to obtain a dynamical observer
  %for the original observer design problem.
  %For any controlled DAE, we construct an LTI system, whose
  %outputs correspond to the input and state trajectories of that DAE.
  %This construction works for any DAE, no assumptions such as regularity, existence of
  %a solution, etc. are required.
  The translation of the dual control problem to an LTI control problem
  relies on linear geometric control theory
  \cite{TrentelmanBook,GeomBook2}: the state and input trajectories of the DAE correspond
  to trajectories of an LTI restricted to its largest output zeroing subspace.
   %Using geometric control, we find $\widehat U(t_1,\cdot)$ in a feed-back form.
   %More precisely, by restricting a suitable linear ODE to its weak observable subspace,
   %we can construct a one-to-one and linear correspondence between the trajectories of an ODE
   %and those of a DAE.
  To sum up, in this paper we solve
  the \textbf{(1)} minimax estimation problem, and the \textbf{(2)} infinite horizon optimal
  control problem for DAEs. In addition, we do no impose a-priori restrictions on $F$ and $A$.
% This represents the main contribution of this paper.

 \textbf{Related work}
   To the best of our knowledge, the results of this paper are new.
   The literature on DAE is vast, but most of the papers concentrate on regular DAEs.
   The papers \cite{Xu2007,Darouch2009} are probably the closest 
    to the current paper. However,
   unlike in \cite{Xu2007}, we allow non-regular DAEs, and unlike
   \cite{Darouch2009}, we do not require impulsive observability.
   In addition, the solution methods are also very different.
   %Moreover, unlike \cite{Xu2007,Darouch2009}, 
   %we construct observers whose worst-case error is minimal.
   %The main novelty of the paper lies
   %in formulating and solving the minimax estimation problem and  the
   %infinite horizon linear quadratic optimal control problems for general DAEs.
   The finite horizon minimax estimation problem and the corresponding optimal control
   problem for general DAEs was presented in~\cite{Zhuk2012sysid}.
   A different way of representing solutions of DAEs as outputs of a LTI were
   presented in \cite{Zhuk2012sysid} too.
   %This correspondance allows us to transform the LQ problem for DAEs into the LQ problem for
   %LTI system. As a result we represent $\widehat U(t_1,\cdot)$ as an output of the LTV-system
   %that yields the minimax estimate $\widehat{\mathcal O}_{t_1}(y)$ in the recursive form.
   %The latter is then used to construct the minimax estimate
    %$\widehat{\mathcal O}_{\infty}(y)$ for the infinite-horizon problem. The minimax estimate
    We note that a feed-back control for finite and
    infinite-horizon LQ control problems with stationary DAE constraints was constructed
    in~\cite{BenderLaub1987IEEETAC} assuming that the matrix pencil $F-\lambda A$ was regular.
    It was mentioned in~\cite{Zhuk2012sysid} that transformation of
    DAE  into Weierstrass canonical form may require taking derivative of the model error $f$,
    which, in turn, leads to restriction of the
    admissible class of model errors. In contrast, our approach is valid for $L^2$-model
    errors, which makes it more attractive for applications.
      %Another way to address LQ control problem with DAE constraint is by application of null-output concept from geometric control theory~\cite{TrentelmanBook}. {\bf TODO: Misha, could u briefly talk about this here?}.
     Generalized Kalman duality principle for non-stationary DAEs with non-ellipsoidal
     uncertainty description was introduced in~\cite{Zhuk2012AMO} where it was applied to get
     a sub-optimal infinite-horizon observer.
     The infinite-horizon LQ control problem for non-regular DAE was also addressed in \cite{LQGNonReg}, but unlike this paper, there it is assumed that the DAE has a solution from any initial state. Optimal control of non-linear and time-varying DAEs was also addressed in
the literature. Without claiming completeness we mention \cite{Kurina2007,MehramnnKunk2008}.

\textbf{Outline of the paper}
This paper is organized as follows. Subsection~\ref{sec:notation} contains notations, section~\ref{sec:problem-statement} describes the mathematical problem statement, section~\ref{sec:main} presents the main results of the paper.

\subsection{Notation}
\label{sec:notation}
%$\mathbb R^n$ denotes the $n$-dimensional Euclidean space;
%$L^2(t0,t_1, \mathbb R^m)$ denotes a space of square-integrable functions with values in $\mathbb R^m$ (in what
%follows we will often write $L^2(t_0,t_1)$ referring $L^2(t_0,t_1,\mathbb R^k)$ where the dimension $k$ will be defined by the context);
%$\mathbb H_1(t_0,t_1,\mathbb R^m)$ denotes a space of absolutely continuous functions with $\mathbb L^2(t_0,t_1)$-derivative;
%The prime $'$ denotes the operation of taking the adjoint: $L'$
%denotes adjoint operator, $F'$ denotes the transposed matrix;
%$x^Ty$ denotes the inner product of vectors $x,y\in\mathbb R^n$, $\|x\|^2:=x^Tx$;
$S>0$ means $x^TSx>0$ for all $x\in\mathbb R^n$;
$F^+$ denotes the pseudoinverse matrix.
 Let $I$ be either a finite interval $[0,t]$ or the infinite time axis $I=[0,+\infty)$.
 We will denote by $L^2(I,\mathbb{R}^{n})$, $L^2_{loc}(I,\mathbb{R}^{n})$
 the sets of
 all square-integrable, and locally square integrable
 functions $f:I \rightarrow \mathbb{R}^{n}$ respectively.
 Recall that a function is locally square integrable, if its restriction to any compact
 interval is square integrable. If $I$ is a compact interval, then $L^2_{loc}(I,\mathbb{R}^n)=L^2(I,\mathbb{R}^n)$.
 %We denote by $L^2_{loc}(\mathbb{R}^{n})$ the class of locally square integrable
 %functions defined $[0,+\infty)$, i.e. $f:[0,+\infty) \rightarrow \mathbb{R}^n$, and for any
 %$T > 0$, the restriction of $f$ to $[0,T]$ belongs to $L^2([0,T],\mathbb{R}^n)$.
 If $\mathbb{R}^n$ is clear from the context and $I=[0,t]$, $t > 0$, we will use the notation
 $L^2(0,t)$ and $L^2_{loc}(0,t)$ respectively.
 If $f$ is a function, and $A$ is a subset of its domain, we denote by $f|_{A}$ the
 restriction of $f$ to $A$.
 We denote by $I_n$ the $n \times n$ identity matrix.

 %If $E \in \mathbb{R}^{n \times n}$, then denote by
 %$A_{E}(I,\mathbb{R}^n)$ the set of all functions
 %$g \in L^2(I,\mathbb{R}^n)$ such that $Eg \in A(I,\mathbb{R}^{n})$.
 %Similarly, we denote by
 %$A_{E}^{loc}(I,\mathbb{R}^n)$ the set of all functions
 %$g \in L^2_{loc}(I,\mathbb{R}^n)$ such that $Eg \in A([0,+\infty),\mathbb{R}^n)$.

\section{Problem statement}
\label{sec:problem-statement}
Assume that $x(t)\in\mathbb R^n$ and $y(t)\in\mathbb R^p$ represent the state vector and output of the following DAE:
\begin{equation}
\label{eq:dae_output}
\begin{split}
      & \dfrac{d(Fx)}{dt}=Ax(t)+f(t)\,,\quad Fx(0)=x_0\,, \\
      & y(t)=Hx(t)+\eta(t)\,,
\end{split}
\end{equation}
where $F,A\in\mathbb R^{ n \times n}$, $H\in\mathbb R^{p\times n}$, and $f(t)\in\mathbb R^n$, 
$\eta(t)\in\mathbb R^p$ stand for the model error and output noise respectively. 
In this paper we consider the following functional class for DAE's solutions: if $x$ is a 
solution on some finite interval $I=[0,t_1]$ or 
infinite interval $I=(0,+\infty)$, then $x \in L^2_{loc}(I)$, and $Fx$ is absolutely continuous. 
This allows to consider a state vector $x(t)$ with a non-differentiable part belonging to the null-space of $F$. 
We refer the reader to~\cite{Zhuk2012AMO} for further discussion.

In what follows we assume that for any initial condition $x_0$ and any time interval
$I=[0,t_1]$, $t_1 < +\infty$, model error $f$ and output noise $\eta$ are unknown and belong to the given
ellipsoidal bounding set $\mathscr E(t_1):=\{(x_0,f,\eta) \in \mathbb{R}^{n} \times L^2(I,\mathbb{R}^n) \times L^2(I,\mathbb{R}^{p}):\rho(x_0,f,\eta,t_1)\le 1\}$, where
\begin{equation}
  \label{eq:rhox0feta}
\rho(x_0,f,\eta,t_1):=x_0^T Q_0 x_0 + \int_{0}^{t_1} f^TQf + \eta^TR\eta dt\,,
\end{equation}
and $Q_0,Q(t)\in\mathbb R^{n \times n}$, $Q_0=Q_0^T>0$, $Q=Q^T>0$, $R
\in\mathbb R^{p\times p}$, $R^T=R>0$. In other words, we assume that
the triple $(x_0,f,\eta)$ belongs to the unit ball defined by the norm
$\rho$.

First, we study the state estimation problem for finite time interval $[0,t_1]$.
Our aim is to construct the estimate of the linear function of
the state vector $\ell^TFx(t_1)$, $\ell \in \mathbb{R}^n$, given the output $y(t)$ of
\eqref{eq:dae_output}, $t \in
[0,t_1]$.
Following{~\cite{Astrem2006}} we will be looking for an
estimate in the class of linear functionals
$$
\mathcal O_{U,t_1}(y)=\int_{0}^{t_1} y^T(s)U(s)dt\,,
$$
$U \in L^2(0,t_1)$.
Such linear functionals represent linear estimates of a state component
$\ell^TFx(t_1)$ based on past outputs $y$.
%It is easy to see that $\mathcal{O}_{U,t_1}$ is uniquely determined by $U$.
We will call functions $U \in L^2(0,t_1)$ \emph{finite horizon observers}.
With each such observer $U$ we will associate an observation error defined as follows.
%In other words, $U(t_1,\cdot)$ maps $R$ into a set of $L^2_{loc}(0,+\infty)$-fucntions on $(0,t_1)$ and $U(\infty,\cdot)$ coincides with a function from $L^2_{loc}(0,+\infty)$.
%\begin{Definition}\label{d:1}
%Assume $t_1<+\infty$. To any observer $U \in L^2(0,t_2)$ and vector $\ell\in\mathbb R^m$ we assign a worst-case estimation error:
\begin{equation*}
  \begin{split}
    \sigma(U,t_1,\ell):=\sup_{(x_0,f,\eta)\in \mathscr E(t_1)}(\ell^TFx(t_1)-\mathcal O_{U,t_1}(y))^2\,.
  \end{split}
\end{equation*}
%The estimate $\widehat{U} \in L^2(0,t_1)$
%$\widehat{\mathcal O}_{t_1}(y)$ corresponding to
%$\widehat{U}(t_1,.) \in L^2([0,t_1])$
%is said to be minimax ($\ell$-estimate), if
%$$
%\sigma(\widehat{U},t_1,\ell) = \inf_{U \in L^2(0,t_1)}\sigma(U,t_1,\ell)\,.
%$$ The minimax error ($\ell$-error is represented by $\hat\sigma(t_1,\ell):=\sigma(\widehat{U},t_1,\ell)$.
%\end{Definition}
%\begin{Remark}
% Notice that there is a one-to-one
% correspondence between the restriction of $U(t_1,\cdot) \in L^2_{loc}(0,+\infty)$  to
% $[0,t_1]$ and the map $O_{t_1}(y)$.
% Hence, for $t_1 < +\infty$, it would be more natural to define
% $U(t_1,\cdot)$ as an element of $L^2(0,t_1)$. However, for $t_1=+\infty$, the setting
% formulated above seems more natural.
%determines a linear functional
 %$O_{U(t_1,\cdot}}: L^{2}(0,t_1) \ni y \mapsto O_{t_1}(y)$ uniquely. The converse is true in
 %the sense that
 %if for some $\widehat{U}(t_1,\cdot) \in L^2_{loc}(0,+\infty)$,
 %and $O_{U(t_1,\cdot)}=O_{\widehat{U}(t_1,\cdot)}$, then
 %$\widehat{U}(t_1,s)=U(t_1,s)$ for all $s \in [0,t_1]$.  That  is,
 %$O_{U(t_1,\cdot)}$ uniquely determines the restriction of $U(t_1,.)$ to
 %$[0,t_1]$.
%\end{Remark}
 The observation error $\sigma(U,\ell,t_1)$ represents the biggest estimation
 error of $\ell^TFx(t_1)$ which can be produced by the observer $U$, if we
 assume that the initial state and the noise belong to $\mathscr E(t_1)$.

 So far, we have defined observers which act on finite time intervals. Next, we
 will define an analogous concept for the whole time axis $[0,+\infty)$.
 %In order to formulate the concept of an infinite horizin observer, we will need
 %the following notation.
\begin{Definition}[Infinite horizon observers]
  Denote by $\mathscr{F}$ the set of all maps
  $U:\{ (t_1,s) \mid t_1 > 0, s \in [0,t_1]\} \rightarrow \mathbb{R}^p$ such that for every
 $t_1 > 0$,
  the map $U(t_1,\cdot): [0,t_1] \ni s \mapsto U(t_1,s)$ belongs to $L^2(0,t_1)$.

 An element $U \in \mathscr{F}$ will be called an \emph{infinite horizon observer}.
 If $y \in L^2_{loc}(I,\mathbb{R}^{p})$, $I=[0,t_1]$, $t_1 > 0$ or $I=[0,+\infty)$, then
 the result of applying $U$ to $y$ is a function
 $O_{U}(y):I \rightarrow \mathbb{R}$ defined by
 \[ \forall t \in I: O_{U}(y)(t)=O_{U(t,\cdot),t}(y)=\int_0^{t} U^T(t,s)y(s)ds. \]

 The worst-case error for $U \in \mathscr{F}$ is defined as
\[ \sigma(U,\ell):=\limsup_{t_1 \rightarrow \infty}\sigma(U(t_1,\cdot),t_1,\ell). \]
 %provided the limit exists.
%$$ The minimax estimate $\widehat{U}$ is defined as the element
 %$\widehat{U} \in mathscr{F}(\mathbb{R}^{p}$,  such that
%$$ \sigma(\widehat{U},\infty,\ell):=\inf_{U \in \mathscr{F}(\mathbb{R}^p)} \sup_{t_1>0}\sigma(U(t_1,\cdot),t_1,\ell)\,.  $$
\end{Definition}
 Intuitively, an infinite horizon observer is just a collection of finite horizon observers,
 one for each time interval. It maps any output defined on some interval (finite or infinite)
 to an estimate of a component of the corresponding state trajectory.
 The worst case error of an infinite horizon observer
 represents the largest asymptotic error of estimating $\ell^TFx(t)$ as $t \rightarrow \infty$.
 %that the initial state and the noises acting on the intervall $[0,t_1]$ belong to
 %$\mathscr{E}(t_1)$. This interpretation does not require that the infinite horizon
 %observer acts only output trajectories which are defined on the whole time axis.

 The effect of applying an infinite horizon observer
 $U \in \mathscr{F}$ to an  output $y \in L^2_{loc}([0,+\infty),\mathbb{R}^{p})$
 of the system \eqref{eq:dae_output} can be described as follows.
 Assume that $y$
 corresponds to some initial state $x_0$ and noises $f$ and $\eta$ such that
 %for all $t_1 > 0$, $(x_0,f|_{[0,t_1]},\eta|_{[0,t_1]}) \in \mathscr{E}(t_1)$.
 \[ x_0Q_0x_0 + \int_{0}^{+\infty} f^T(t)Qf(t) + \eta^T(t)R\eta(t) dt \le 1. \]
 The latter restriction can equivalently be stated as $(x_0,f|_{[0,t_1]},\eta|_{[0,t_1]}) \in \mathscr{E}(t_1)$, $\forall t_1 > 0$.
 Assume that $x$ is the state trajectory corresponding to $y$. Then
 $O_{U}(y)$ represents an estimate of $\ell^TFx$ and the estimation error is bounded
 from above by $\sigma(U,\ell)$ in the limit, i.e.
 for every $\epsilon > 0$ there exists $T > 0$ such that for all $t > T$
 \[
    \sigma(U,\ell) + \epsilon >  (\ell^TFx(t)-\mathbf{O}_{U}(y)(t))^2
 \]

So far we have defined observers as linear maps mapping past outputs to state estimates.
For practical purposes it is desirable that the observer is represented by a stable LTI system.
\begin{Definition}
 The observer $U \in \mathscr{F}$ can be represented by
 a stable linear system, if there exists $A_o \in \mathbb{R}^{r \times r},B_o \in \mathbb{R}^{r \times p}, C_o\in \mathbb{R}^{1 \times r}$ such that $A_o$ is stable and
for any $y \in L^2_{loc}(I)$, $I=[0,t_1]$, $t_1 > 0$ or $I=[0,+\infty)$, the estimate $\mathcal{O}_{U}(y)$ is the output of the LTI system below:
 \begin{align*}
   & \dot s(t) = A_os(t)+B_oy(t) \mbox{,\ \ } s(0)=0 \\
   & \forall t \in I: \mathcal{O}_{U}(y)(t)=C_os(t).
 \end{align*}
  %The map $\widehat{U}$ will be called a minimax estimate of $\ell^Tx$.
  The system $\mathscr{O}_{U}=(A_o,B_o,C_o)$ is called a
  \emph{dynamical observer} associated with $U$.
\end{Definition}
In addition, we would like to find observers with the smallest possible
worst case observation error.
These two considerations prompt us to define the minimax observer design problem
as follows.
\begin{Problem}[Minimax observer design]
\label{problem:obs}
  Find an observer $\widehat{U} \in \mathscr{F}$ such that
   \begin{equation}
   \label{problem:obs:eq2}
      \sigma(\widehat{U},\ell) = \inf_{U \in \mathscr{F}} \sigma(U,\ell)<+\infty
   \end{equation}
  and $\widehat{U}$ can be represented by a stable linear system. In what follows we will refer to such $\widehat{U} \in \mathscr{F}$ as minimax observer.
%\end{enumerate}
\end{Problem}
%We note that $\sigma(U,\ell)$ has the following meaning: by duality we represent $\sigma(U,\ell)$ as a quadratic cost function defined over solution of the DAE adjoint to~\eqref{eq:dae_state}. This corresponds to the dual control problem which is well defined for the infinite-horizon case. We elaborate on this in the following sections.

\section{Main results}
\label{sec:main}
In this section we present our main result: minimax observer for the infinite horizon case.
 First, in \S \ref{sec:dual} we present the dual \emph{optimal control problem}
 for infinite horizon case.
 This dual control problem which we are going to formulate is interesting itself.
 In order to solve the optimal control problem, we will use
 the concept of output zeroing space from the geometric control.
This technique allows us to construct an LTI system whose outputs are solutions of the
original DAE.  This will be discussed in \S \ref{dae:geo}.
In \S \ref{dae:opt} we reformulate the dual
optimal control problem as a linear quadratic infinite
horizon control problem for LTIs. The solution of the latter problem yields a solution to
the dual control problem.
Finally, in \S \ref{sec:obs} we present the formulas for the minimax observer and discuss the conditions for its existence.

\subsection{Dual control problem}
\label{sec:dual}
 We will start with formulating an optimal control problem for
 DAEs. Later on, we will show that the solution of this
 control problem yields a solution to the minimax observer design problem.
 %However, the control problem we are going to formulate is interesting on its
 %own right.
  Consider the DAE $\Sigma$:
 \begin{equation}
 \label{dae:sys}
 %\begin{split}
   \dfrac{dEx}{dt} = \hat{A}x(t) + \hat{B}u(t) \mbox{ and } Ex(0)=Ex_0.
 %\end{split}\right.
\end{equation}
Here $x_0 \in \mathbb{R}^{n}$ is a fixed initial state and
$\hat{A},E \in \mathbb{R}^{n \times n}$, $\hat{B} \in \mathbb{R}^{n \times m}$.
%\begin{Notation}
%In the sequel, we will identify $\Sigma$ with the tuple $\Sigma=(E,\hat{A},\hat{B},x_0)$.
%\end{Notation}
\begin{Notation}[$\mathscr{D}_{x_0}(t_1)$ and $\mathscr{D}_{x_0}(\infty)$].
\label{sol:not}
 For any $t_1 \in [0,+\infty]$ denote by $I$ the interval $[0,t_1] \cap [0,+\infty)$
 and denote by
  \( \mathscr{D}_{x_0}(t_1) \) the set of all pairs
\( (x,u) \in L^2_{loc}(I,\mathbb{R}^n) \times L^2_{loc}(I,\mathbb{R}^m) \) such that
 $Fx$ is absolutely continuous and $(x,u)$ satisfy \eqref{dae:sys}.
 %Likewise, define
%  \(
      %\mathscr{D}(\infty) \)
  %as the set of all tuples $(x,u,d) \in A_E^{loc}(\mathbb{R}^n) \times L^2_{loc}(\mathbb{R}^{m}) \times L^2_{loc}(\mathbb{R}^n)$ such that
 %$(x,u)$ satisfies \eqref{dae:sys} and $Ed(t)=0$ for all $t \ge 0$.
\end{Notation}
 Note that we did not assume that the DAE is regular, and hence
 there may exist initial states $x_0$ such that
 $\mathcal{D}_{x_0}(t_1)$ is empty for some $t_1 \in [0,+\infty]$.
%Now we are ready to formulate the optimal control problem  announced above.
 \begin{Problem}[Optimal control problem:]
\label{opt:contr:def}
 Take $R\in\mathbb R^{m\times m},Q,Q_0\in\mathbb R^{n\times n}$ and assume that $R > 0, Q> 0$ $Q_0 \ge 0$.
 For any initial state $x_0 \in \mathbb{R}^n$,
 and any trajectory $(x,u) \in \mathscr{D}_{x_0}(t)$, $t > t_1$ define the cost
 functional
 \begin{equation}
 \label{opt1.1}
 \begin{split}
    & J(x,u,t_1) = x(t_1)^TE^TQ_0Ex(t_1)+\\
     & + \int_0^{t_1} (x^T(s)Qx(s)+u^T(s)Ru(s)) ds.
 \end{split}
 \end{equation}
  For every $(x,u) \in \mathscr{D}(\infty)$, define
 \[ J(x,u)=\limsup_{t_1 \rightarrow \infty} J(x,u,t_1)\,. \]

  The infinite horizon optimal control problem for \eqref{dae:sys}
  is the problem of finding a tuple of matrices $(A_c,B_c,C_x,C_u)$ such that
 $A_c \in \mathbb{R}^{r \times r},B_c \in \mathbb{R}^{r \times n}$ ,
%$C=\begin{bmatrix} C_x, & C_u \end{bmatrix}$,
 $C_x \in \mathbb{R}^{n \times r}$, $C_u \in \mathbb{R}^{m \times r}$,
  $A_c$ is a stable matrix,
  $B_cEC_x=I_{r}$, and
  for any $x_0 \in \mathbb{R}^{n}$ such that
  $\mathcal{D}_{x_0}(\infty) \ne \emptyset$,
 the output of the system
 \begin{equation}
 \label{dyncontr:eq}
  \begin{split}
   \dot s(t) = A_cs(t) \mbox{ and } s(0)=B_cEx_0 \\
    x^{*}(t)=C_x s(t) \mbox{ and } u^{*}(t)=C_{u}s(t),
  \end{split}
\end{equation}
is such that $(x^{*},u^{*}) \in \mathscr{D}_{x_0}(\infty)$, and
  \begin{equation}
  \label{opt:eq2}
     J(x^{*},u^{*}) = \limsup_{t_1 \rightarrow \infty } \inf_{(x,u) \in \mathscr{D}_{x_0}(t_1)} J(x,u,t_1).
 \end{equation}
  %and $(x^{*},u^{*})$ is the output of an autonomous stable linear system, i.e.
  %there exists matrices
 The tuple $\mathscr{C}^{*}=(A_c,B_c,C_x,C_u)$ will be called the
 \emph{dynamic controller} which solves the optimal control problem.
 For each $x_0$, the pair $(x^{*},u^{*})$ will be called the solution of
 the optimal control problem for the initial state $x_0$.

  We will denote infinite horizon control problems above
  by $\mathcal{C}(E,\hat{A},\hat{B},Q,R,Q_0)$.
\end{Problem}
Note that the dynamic controller which generates the solutions of
the optimal control problem does not depend on the initial condition, in fact,
the dynamical controller generates a solution for any initial condition, for
which the DAE admits a solution on the whole time axis. % defined on the whole time interval.

\begin{Remark}
 The proposed formulation of the infinite horizon control problem is not necessarily
 the most natural one. We could have also required the
 $(x^{*},u^{*}) \in \mathcal{D}(\infty)$ to satisfy
 \( J(x^{*},u^{*}) = \inf_{(x,u)\in \mathscr{D}(\infty)} J(x,u)\).
 %The latter means that the cost induced by $(x^{*},u^{*})$ is the
 %smallest among all the trajectories $(x,u)$ which are defined on the whole
 %time axis. 
 It is easy to see that formulation above implies that
 \( J(x^{*},u^{*}) = \inf_{(x,u) \in \mathscr{D}(\infty)} J(x,u)\).
 Another option could have been to use limit instead of
 $\limsup$ in the definition of $J(x^{*},u^{*})$ and in
 \eqref{opt:eq2}.
 In fact, the solution we are going to present remains a solution if we replace
 $\limsup$ by limits.
%% \begin{equation}
%% \label{eq:op21}
%% \begin{split}
%%   & \lim_{t_1 \rightarrow \infty} J(x^{*},u^{*},t_1)=
%%   \lim_{t_1 \rightarrow \infty} \inf_{(x,u) \in \mathscr{D}(t_1)}
%%    J(x,u,t_1)  \mbox{ or } \\
%%   & \lim_{t_1 \rightarrow \infty} J(x^{*},u^{*},t_1) =
%%    \inf_{(x,u) \in \mathscr{D}(\infty)}
%%     \lim_{t_1 \rightarrow \infty} J(x,u,t_1)
%% \end{split}
%% \end{equation}
%% In fact, the solution we will present later on will not only satisfy
%% \eqref{opt:eq2}, but also both equations of \eqref{eq:op21}
\end{Remark}
\begin{Remark}[Solution as feedback]
  In our case, the optimal control law $u^{*}$ can be interpreted as
  a state feedback.
  %a linear system. However, we will show that the optimal control law can in fact
  %be represented as a state feedback of the DAE. It is easy to see that
  If $\mathcal{C}^{*}=(A_c,B_c,C_x,C_u)$ is the optimal dynamical controller and $x_0 \in \mathbb{R}^n$,
  and $(x^{*},u^{*})$ is as in \eqref{dyncontr:eq}, then
  $s(t)=B_cEC_{x}s(t)=B_cEx^{*}(t)$ and thus
  \( u^{*}(t)=B_cEx^{*}(t). \)
  Note, however, that for DAEs the feedback law does not determine the control
  input uniquely, since even autonomous DAEs may admit several solutions starting
  from the same initial state.
  If the DAE has at most one solution from any
  initial state, in particular, if the DAE is regular, then the feedback law
  above determines the optimal trajectory $x^{*}$ uniquely.
\end{Remark}
%\begin{Remark}[Solution as control law]
% As it was noted above, even for fixed feedback control law $u^{*}=BEx^{*}(t)$
% and initial state,
% the DAE may admit several solutions. Hence, calling the optimal input $u^{*}$
% control law is justified only if the DAE has at most one solution from each
% initial state.
% Note that for solving the observer design
% problem the existence of several trajectories will not be a problem.
%\end{Remark}
\begin{Remark}[Closed-loop stability]
 Since the optimal state trajectory $x^{*}$ is the output of
 a stable LTI, $\lim_{t \rightarrow \infty} x^{*}(t)=0$.
 Hence, if the DAE admits at most one  solution from any initial state, then
 the closed-loop system
 %\( \frac{d(Ex(t))}{dt} =  (\hat{A}+\hat{B}B_cE)x(t) \)
 is globally asymptotically stable, i.e.
 for any initial state the corresponding solution converges to zero.
 %the closed-loop DAE has a solution on $[0,+\infty)$, the
 %corresponding state trajectory converges to zero.
 %solution $x$ defined on $[0,+\infty)$, such that
 %$\lim_{t \rightarrow \infty} x(t)=0$.
 %If the original DAE has the property that from any initial state it has at most
 %one solution, then the closed-loop DAE will be globally asymptotically stable
 %in the usual sense.
\end{Remark}
 Now we are ready to present the relationship between Problem \ref{opt:contr:def}
 and Problem \ref{problem:obs}.
 \begin{Definition}[Dual control problem]
  The dual control problem for the observer design problem
  is the control
  problem
  $\mathcal{C}(F^T,A^T,-H^T,Q^{-1},R^{-1},\bar{Q}_0)$, 
  where
  \begin{align*}
     & \bar{Q}_0=({F^T}^+ z(0)-\LOPT)^TQ_0^{-1} ({F^T}^+z(0)-\LOPT).
  \end{align*}
   Here $\LOPT$ is defined as follows.
  Let $r=\Rank F^T$ and $U \in \mathbb{R}^{n \times (n-r)}$ such that
  $\IM U = \ker F^T$ and define $\LOPT=U(U^TQ^{-1}_0U)^{-1}U^TQ^{-1}_0{F^T}^{+}$.
 \end{Definition}
 \begin{Theorem}[Duality]
 \label{p:5}
  Let $\mathscr{C}_{u^{*}}=(A_c,B_c,C_x,C_u)$
    %\textbf{$A$ is confusing here cause it has been used before in DAE definition}
   be the dynamic controller solving
  the dual control problem. Let $(x^{*},u^{*})$ be the corresponding solution
  of the optimal control problem for $x_0=\ell$.
  Then $\widehat{U}(t_1,s)=u^{*}(t_1-s)$ is the solution of the infinite time horizon
  observer design problem,  and
  \[
    \begin{split}
      & \sigma(\widehat{U},\ell)=J(x^{*},u^{*})= \limsup_{t_1 \rightarrow \infty} \{ {x^{*}}^T(t_1)F\bar{Q}_0F^Tx^{*}(t_1)+\\
           & \int_{0}^{t_1} ({u^{*}}^T(t) R^{-1}u^*(t)+{x^{*}}^T(t)Q^{-1}x^{*}(t))dt\}.
    \end{split}
 \]
 In addition, the dynamical observer
  $\mathscr{O}_{\widehat{U}}$ is of the form
  \begin{align*}
     & \dot s(t) = A^T_cs(t)+C^T_u y(t) \mbox{, } s(0)=0 \\
     & \mathcal{O}_{\hat{U}}(y)(t)=\ell^TFB^T_cs(t)
  \end{align*}
  Moreover, if  %$\lim_{t_1 \rightarrow \infty} x^{*}(t_1)=0$ and
  $y \in L^2_{loc}([0,+\infty),\mathbb{R}^p)$ is the output of \eqref{eq:dae_output} for $f=0$ and $\eta=0$,
  then the estimation error $(\ell^TFx(t)-\mathcal{O}_{\hat{U}}(y)(t))$ converges to zero as
  $t \rightarrow \infty$.
 \end{Theorem}
  Note that the matrices of the observer presented in Theorem \ref{p:5} depend on
  $\ell$ only through the equation $\mathcal{O}_{\hat{U}}(y)(t)=\ell^TFB^T_c s(t)$.
  Hence, if a solution to the dual control problem exists, then it yields an
  observer for any $\ell$, for which the dual DAE
  \( \dfrac{d(F^{T}z(t))}{dt}=A^Tz(t)-H^Tv(t),\quad F^Tz(0)=F^T\ell\)
  has a solution defined on the whole time axis.
  
  Theorem \ref{p:5} implies that existence of a solution of
  the dual control problem is a sufficient condition for existence of a solution 
  for Problem \ref{problem:obs}. 
  In fact, we conjecture that this condition is also a necessary one. 

 \begin{proof}[Proof of Theorem \ref{p:5}]
 Recall from \cite{Zhuk2012sysid} the following duality principle:
\begin{Proposition}%[dual control problem]
\label{p:1}
 Consider the adjoint DAE:
\begin{equation}
    \label{eq:zul}
\dfrac{d(F^{T}z(t))}{dt}=-A^Tz(t)+H^Tv(t),\quad F^Tz(t_1)=F^T\ell\,.
  \end{equation}
\textbf{(1)}
There exists $U \in L^{2}(0,t_1)$ such that $\sigma(U,\ell,t_1) < +\infty$ iff there exists
$z \in L^2(0,t_1)$ and $v \in L^2(0,t_1)$ such that $F^Tz$ is absolutely continuous and
$(z,v)$ satisfies \eqref{eq:zul}. \\
\textbf{(2)}
 Denote by $\mathscr{DD}(t_1)$ is the set of all tuples $(z,d,v) \in L^2(0,t_1) \times \mathbb{R}^n \times L^2(0,t_1)$ such that $F^Tz$ is absolutely continuous and  $(z,v)$ satisfy \eqref{eq:zul} and
$F^Td=0$.
 For all  $(z,d,v) \in \mathscr{DD}(t_1)$, define
\begin{equation}
  \label{eq:umin}
  \begin{split}
&\mathscr I(z,d,v,t_1):=
\int_{0}^{t_1}(v^T(t)R^{-1}v(t)+z^T(t)Q^{-1}z(t))dt\\%(:=\mathscr I_1(z,u))\\
&+({F^T}^+F^T z(0)-d)^TQ_0^{-1} ({F^T}^+F^T z(0)-d)\ \\
  \end{split}
 \end{equation}
%If $\inf_{U \in L_2(0,t_1)} \sigma(U,\ell,t_1)<+\infty$, then
For any $U \in L_2(0,t_1)$ such that $\sigma(U,\ell,t_1) <+\infty$,
\[
   \sigma(U,\ell,t_1)=\inf_{(z,d,v) \in \mathcal{DD}(t_1), v=U} \mathscr{I}(z,d,v,t_1).
\]
\textbf{(3)}
Moreover, if $\inf_{U \in L_2(0,t_1)} \sigma(U,\ell,t_1)<+\infty$, then
there exists $(z^{*},d^{*},\widehat{U}) \in \mathscr{DD}(t_1)$ such that
\begin{equation}
\begin{split}
\label{p:1:eq1}
  & \sigma(\hat{U},\ell,t_1)=
  \inf_{U \in L_2(0,t_1)} \sigma(U,\ell, t_1) =\\
  & = \inf_{(z,d,v) \in \mathscr{DD}(t_1)} \mathscr{I}(z,d,v,t_1) =
     \mathscr{I}(z^{*},d^{*},\hat{U},t_1)
\end{split}
\end{equation}
\end{Proposition}
 Note that in ~\cite{Zhuk2007} it was proved that the DAE adjoint to~\eqref{eq:dae_output} has the form~\eqref{eq:zul}.
 Proposition \ref{p:1} allows us to reduce the problem of minimax
observer design to that of finding an optimal controller.
 To this end, we transform
 slightly the statement of Proposition \ref{p:1}.
 First, we get rid of the component $d$ of the optimization problem
 from Proposition \ref{p:1}.
\begin{Proposition}
\label{prop1}
  Let $(z,v)$ be a solution of \eqref{eq:zul} such that
  $z \in L^2(0,t_1)$, $v \in L^2(0,t_2)$, $F^Tz$ is  absolutely continous.
  Then
  $\inf_{d \in \mathbb{R}^{n}, F^Td=0} \mathscr{I}(z,d,v,t_1)=\mathscr{I}(z,\LOPT(F^{T}z(0)),v,t_1)$.
  %Let $(z^{*},d^{*},v^{*}) \in \mathscr{DD}(t_1)$ such that
  %$\mathscr{I}(z^{*},d^{*},v^{*},t_1)=\inf_{(z,d,v) \in \mathscr{DD}(t_1), v=v^{*}} \mathscr{I}(z,d,v,t_1)$.
  %Recall the definition of $\LOPT$.
  %Then $d^{*}=\LOPT(F^{T}z^{*}(0))$.
 \end{Proposition}
Hence, instead of the cost function $\mathscr{I}(z,d,v,t_1)$, it will be enough to consider
 the cost function:
 \[
  \begin{split}
   & \mathscr{I}(z,v,t_1) = z(0)F\bar{Q}_0F^Tz(0)+\\
   & + \int_{0}^{t_1}(v^T(t)R^{-1}v(t)+z^T(t)Q^{-1}z(t))dt
  \end{split}
 \]
 Next, we replace the DAE \eqref{eq:zul} by the DAE of the dual control problem:
 %boundary condition is formulated for the initial state as opposed to the terminal state/
 %To this end, consider the DAE of the dual control problem.
\begin{equation}
\label{p:12:eq1}
   \frac{d(F^Tx(t))}{dt}=A^Tx(t)-H^Tu(t) \mbox{ and } F^Tx(0)=F^T\ell.
\end{equation}
The DAE \eqref{p:12:eq1} is obtained from \eqref{eq:zul} by reversing the time.
In order to present the result precisely, we introduce the following notation.
\begin{Notation}[$\delta_{t_1}$]
 If $r$ is a map defined on $[0,t_1]$, then we denote by $\delta_{t_1}(r)$
 the map $\delta_{t_1}(r)(t)=r(t_1-t)$, $t \in [0,t_1]$.
\end{Notation}
Then $(x,u)$, is a solution of \eqref{p:12:eq1} such that
$x \in L^2(0,t_1)$, $Fx$ is absolutely continuous and $u \in L^2(0,t_1)$,
if and only if $(z,v)=(\delta_{t_1}(x),\delta_1(u))$
 %defined by
%\[ \delta_{t_1}(x)(t)=x(t_1-t) \mbox{ and } \delta_{t_1}(u)(t)=u(t_1-t),  \]
is a solution  of \eqref{eq:zul}.
%\highlight{$x(t_1-t)$ is not well defined!}\textbf{Consider $F^T x(t_1-t)$???}.

Consider now the dual control problem, and recall that
\[
 \begin{split}
   & J(x,u,t_1)=x^T(t_1)F\bar{Q}_0F^Tx(t_1) +  \\
    & \int_{0}^{t_1} (u^T(t) R^{-1}u(t)+{x}^T(t)Q^{-1}x(t))dt.
 \end{split}
\]
In addition, recall from Notation \ref{sol:not}
that $\mathscr{D}_{\ell}(t_1)$ and $\mathscr{D}_{\ell}(\infty)$
are the sets of solutions $(x,u)$ of \eqref{p:12:eq1} defined on the
interval $[0,t_1]$ and $[0,+\infty)$ respectively.
It is easy to see that
\[ J(x,u,t_1)=\mathscr{I}(\delta_{t_1}(x),\delta_{t_1}(u),t_1)\,. \]
Hence,  Proposition \ref{p:1} can be reformulated as follows.
\begin{Proposition}
\label{p:2}
 There exists $U \in L^2(0,t_1)$ such that $\sigma(U,\ell,t_1) < +\infty$, if
 there exists a solution $(x,u) \in \mathscr{D}_{\ell}(t_1)$ such that
 $\delta_{t_1}(u)=U$.
 If  $U \in L_2(0,t_1)$ is such that $\sigma(U,\ell,t_1) <+\infty$, then
 \[ \sigma(U,\ell,t_1) = \inf_{(x,u) \in \mathcal{D}_{\ell}(t_1), \delta_{t_1}(u)=U } J(x,u,t_1),
\]
 There exists a solution $\hat{U} \in L^2(0,t_1)$ such that
 $\sigma(\hat{U},\ell,t_1)=\inf_{U \in L^2(0,t_1)} \sigma(U,\ell,t_1) < +\infty$,
 %on $[0,t_1]$, if and only if
 %$+\infty > \sigma(\hat{U},\ell,t_1)=\inf_{U \in L_2(0,t_1)} \sigma(U,l,t_1)$, iff
 iff there exists $(x^{*},u^{*}) \in \mathscr{D}_{\ell}(t_1)$ such that
 %$u^{*}$ is an output of a time-varying linear system,
 \[ J(z^{*},u^{*},t_1) = \inf_{(x,u) \in \mathscr{D}_{\ell}(t_1)} J(x,u,t_1), \]
Then $\widehat{U}$ can be chosen as $\widehat{U}(t)=\delta_{t_1}(u^{*})$.
and
\[ \sigma(\hat{U},\ell,t_1)=J(x^{*},u^{*},t_1). \]
\end{Proposition}
     We are now ready to conclude the proof of the theorem.
     Suppose $(x^{*},u^{*})$ is the solution of the dual control problem.
     Since $(x^{*},u^{*}) \in \mathscr{D}_{\ell}(t_1)$ for all $t_1$, Proposition \ref{p:2}
     yields that 
    %$J(x^{*},u^{*},t_1) < +\infty$.
     %It then follows that for all $t_1 > 0$,
     $\inf_{(x,u) \in \mathscr{D}_{\ell}(t_1)} J(x,u,t_1) < + \infty$.
     From Proposition \ref{p:2} it follows that
     $\inf_{v \in L^2(0,t_1)} \sigma(v,\ell,t_1)=\inf_{(x,u) \in \mathscr{D}_{\ell}(t_1)} J(x,u,t_1) < +\infty$.
     Let $U_{t_1} \in L^2(0,t_1)$ be such that
     $\sigma(U_{t_1},\ell,t_1)=\inf_{v \in L_2(0,t_1)} \sigma(v,\ell,t_1)$.
     From Proposition \ref{p:1} it follows that such $U_{t_1}$ exists for all $t_1 > 0$.
     Define $\bar{U} \in \mathscr{F}$ as
     $\bar{U}(t_1,s)=U_{t_1}(s)$ for all $t_1 > 0$, $s \in [0,t_1]$.
     It then follows that for any $U \in \mathscr{F}$,
     $\sigma(\bar{U}(t_1,\cdot),\ell,t_1) \le \sigma(U(t_1,\cdot),\ell,t_1)$ and hence
     $\sigma(\bar{U},\ell) = \inf_{U \in \mathscr{F}} \sigma(U,\ell) <+\infty$.
     From Proposition \ref{p:2} it then follows that
     $\sigma(\bar{U}(t_1,\cdot),\ell,t_1)=\inf_{v \in L^2(0,t_1)} \sigma(v,\ell,t_1)$ and
     thus
     \[ \sigma(\bar{U},\ell) = \limsup_{t_1 \rightarrow \infty} \inf_{(x,u) \in \mathscr{D}(t_1)} J(x,u,t_1).
     \]
    Define now $\widehat{U} \in \mathscr{F}$ as
    $\widehat{U}(t_1,s)=\delta_{t_1}(u^{*})$, $t_1 > 0$. Then
    $\sigma(\bar{U}(t_1,\cdot),\ell,t_1) \le \sigma(\widehat{U}(t_1,\cdot),\ell,t_1)=\inf_{(x,u^{*}) \in \mathscr{D}_{\ell}(t_1)} J(x,u^{*},t_1) \le J(x^{*},u^{*},t_1)$ and hence
    \[
     \begin{split}
     & \sigma(\bar{U},\ell) \le \sigma(\widehat{U},\ell) \le \limsup_{t_1 \rightarrow \infty} J(x^{*},u^{*},t_1) =\\
     & \limsup_{t_1 \rightarrow \infty} \inf_{(x,u) \in \mathscr{D}_{\ell}(t_1)} J(z,u,t_1) = \sigma(\bar{U},\ell).
     \end{split}
   \]
  and therefore $\hat{U}$ satisfies
  \eqref{problem:obs:eq2}.
   %\[ \sigma(\bar{U},\ell) = \sigma(\hat{U},\ell) = \inf_{U \in \mathscr{F}} \sigma(U,\ell).
  %  \]

  Consider now the dynamical controller $(A_c,B_c,C_u,C_x)$ which is the solution of the
  dual optimal control problem.
  Then $u^{*}(s)=C_{u}e^{A_cs}B_cF^T\ell$ and thus
  $O_{\hat{U}}(y)(t_1)=\int_0^{t_1} \hat{U}^T(t_1,s)y(s)=\int_0^{t_1} \ell^TFB^T_ce^{A^T_c(t_1-s)}C^T_u y(s)ds$. The latter is the output of the linear system $(A^T_c,C_u^T,\ell^TFB^T_c)$ for the
  input $y$ and  the zero initial condition.

  Finally, assume that $y \in L^2_{loc}([0,+\infty), \mathbb{R}^p)$ is the output of the DAE \eqref{eq:dae_output} for the state trajectory $x$ and $f=0$ and $n=0$. Let $(x^{*},u^{*})$ be the solution to the
  dual control problem.
  Consider the derivative of  $r(t)=x^T(t)F^{T}x^{*}(t_1-t)=x^T(t)F^T({F^T}^{+})F^Tx^{*}(t_1-t)$, $t \in [0,t_1]$.
  It follows that
  \[
     \begin{split}
     & \dot r(t) = x^T(t)A^Tx^{*}(t_1-t)-x^T(t)A^T(t_1-t)x^{*}(t_1-t)+\\
     & +x^T(t)Hu^{*}(t_1-t)= {u^{*}}^T(t_1-t)y(t)
     \end{split}
  \]
   and hence
  \[ O_{\widehat{U}}(y)(t_1)=\int_0^{t_1} \dot r(s)ds = x^T(0)F^Tx^{*}(t_1)-x^T(t_1)F^{T}x^{*}(0). \]
  By noticing that $x^{*}(0)={\ell}^T$, it follows that %the estimation error satisfies
  \[ (\ell^TFx(t_1) - O_{\widehat{U}(y)}(t_1)) = x^T(0)F^Tx^{*}(t_1). \]
  Since $F^Tx^{*}(t_1)$ converges to zero as $t_1 \rightarrow \infty$, then the estimation
  error will also converge to zero.
  %It is left to show that $F^Tx^{*}(t_1)$ converges to zero as $t_1 \rightarrow \infty$.

 %\textbf{Necessity}
%  Assume that $\hat{U}$ is a solution of the infinite horizon observer design problem.
%  Assume that $(A,B,C)$ is a stable LTI system such that
%  \begin{align*}
%   & \dot s(t) = As(t) + By(t) \mbox{ and } s(0)=0\\
%   & O_{\hat{U}}(y)(t) = Cs(t).
%  \end{align*}
%  It then follows that $O_{\hat{U}}(y)(t_1)=\int_0^{t_1} \hat{U}^T(t_1,s)y(s)ds=\int_0^{t_1} Ce^{A(t_1-s)}By(s)$ and hence $U(t_1,s)=B^Te^{A^T(t_1-s)}C^T$.
%  Define now $u^{*}(t)=B^Te^{A^Tt}C^T$. We will argue that
%  $u^{*}$ satisfies the conditions of the proposition.
%  From the definition of $u^{*}$ it follows that $u^{*}(s)=\hat{U}(t_1,t_1-s)$.

\end{proof}

%We see now that $\sigma(U(\infty,\cdot),t_1,\ell) = \mathscr I(z,d,U(\infty,\cdot))$ provided~\eqref{eq:zul} has a solution. The definition of the minimax error for the infinite horizon case becomes clear now: one needs to consider the problem~\eqref{eq:umin} for the case $t_1=\infty$. We will make use of this fact in order to construct the minimax estimate for $t_1=\infty$. In the following subsection we suggest one way to construct a linear system such that there is a one-to-one and linear correspondence between its trajectories and those of DAE~\eqref{eq:zul}.

\subsection{DAE systems as solutions to the output zeroing problem}
\label{dae:geo}
 Consider the DAE system \eqref{dae:sys}.
 In this section we will study solution set $\mathscr{D}_{x_0}(t_1)$,
 $t_1\in [0,+\infty]$ of \eqref{dae:sys}.
 %$x \in A_E(I,\mathbb{R}^n), u \in L^2(I,\mathbb{R}^m)$ and
 %for $I=[0,t_1]$, or $x \in A_E^{loc})(\mathbb{R}^n)$ and $u \in L^2_{loc}(\mathbb{R}^m)$
 %for $I-=[0,+\infty)$.
 It is well known that for any fixed $x_0$ and
 $u$, \eqref{dae:sys} may have several solutions or no solution at all.
 In the sequel, we will use the tools of geometric control theory to find a
 subset $\mathcal{X}$ of $\mathbb{R}^{n}$, such that for any $x_0 \in E^{-1}(\mathcal{X})$, $\mathcal{D}_{x_0}(t_1) \ne \emptyset$ for all $t_1 \in [0,+\infty]$.
 Furthermore, we provide a complete characterization of all such solutions as outputs of an LTI  system.
\begin{Theorem}
\label{dae2lin:theo}
 Consider the DAE system \eqref{dae:sys}.
 There exists a linear system $\mathscr{S}=(A_l,B_l,C_l,D_l)$ with
 $A_l \in \mathbb{R}^{\hat{n} \times \hat{n}}$,
 $B_l \in \mathbb{R}^{\hat{n} \times k}$, $C_l \in \mathbb{R}^{(n+m) \times \hat{n}}$
 and $D_l \in \mathbb{R}^{(n+m) \times k}$, $\hat{n} \le n$,  and a linear subspace
 $\mathcal{X} \subseteq \mathbb{R}^n$
 such that the following holds.
 \begin{itemize}
 \item $\Rank D_l=k$.
 \item Consider the partitioning $C_l=\begin{bmatrix} C_s^T, & C_{inp}^T \end{bmatrix}^T$,
       $D_l=\begin{bmatrix} D_{s}^T, & D_{inp}^T \end{bmatrix}^T$,
       $C_s \in \mathbb{R}^{n \times \hat{n}}$,
       $C_{inp} \in \mathbb{R}^{m \times \hat{n}}$,
       $D_{s} \in \mathbb{R}^{n \times k}$,
       $D_{inp} \in \mathbb{R}^{m \times k}$.
       Then $ED_s=0$, $\Rank EC_s = \hat{n}$, $\mathcal{X}=\IM EC_s$.
 %\item $D^TC=0$, $D^TD=I_k$,
 %$\LMAP E\begin{bmatrix} I_n & 0 \\ 0 & 0 \end{bmatrix}C = I_{\hat{n}}$ and $\LMAP E\begin{bmatrix} I_n & 0 \\ 0 & 0 \end{bmatrix}D=0$.
 \item For any $t_1 \in [0,+\infty]$,
       \[ \mathcal{D}_{x_0}(t_1) \ne \emptyset \iff Ex_0 \in \mathcal{X}. \]
 \item
         Define the map
       \( \LMAP = (EC_{s})^{+}:\mathcal{X} \rightarrow \mathbb{R}^{\hat{n}}  \).
       Then $(x,u) \in \mathcal{D}_{x_0}(t_1)$ for some $t_1 \in [0,+\infty]$
        if and only if there exists
       %$(x,u)$ is the output of $\LLIN(\Sigma)$ for
       some input $g \in L^2(I,\mathbb{R}^{k})$, $I=[0,t_1] \cap [0,+\infty)$, such
       that
      \[
        \begin{split}
         & \dot v = A_lv+B_lg \mbox{ and } v(0)=\LMAP(Ex_0) \\
         & x = C_sv + D_sg, \\
         & u = C_{inp}+D_{inp}g,
       \end{split}.
       \]
      Moreover, in this case, the state trajectories $x$ and $v$
      are related as $\LMAP(Ex)=v$.
 %\item $C$ is full column rank, in particular, $(C,A)$ is observable.
 \end{itemize}
\end{Theorem}
%\begin{Definition}
% A linear system $(A,B,C,D)$ described in Theorem \ref{dae2lin:theo} is called
 %a linear system associated with $\Sigma$.
%\end{Definition}
\begin{proof}[Proof of Theorem \ref{dae2lin:theo}]
 There exist suitable nonsingular matrices $S$ and $T$ such that
 \begin{equation}
 \label{dae:sys:tr}
     SET = \begin{bmatrix} I_r & 0 \\
                          0   & 0
          \end{bmatrix},
 \end{equation}
where $r = \Rank E$. Let
 \[  S\hat{A}T=\begin{bmatrix} \widetilde{A} & A_{12} \\
                    A_{21} & A_{22}
     \end{bmatrix} \mbox{,\ \ }
     S\hat{B}=\begin{bmatrix} B_1 \\ B_2 \end{bmatrix}
 \]
 be the decomposition of $A,B$ such that $A_{11} \in \mathbb{R}^{r \times r}$,
 $B_{11} \in \mathbb{R}^{r \times m}$.
 Define
%It then follows that
% \eqref{dae:sys} can be rewritten as
% \begin{equation}
% \label{dae:sys1}
% \begin{split}
%    & \dot \xi_1 = \widetilde{A}\xi_1+A_{12}\xi_2+B_1u \\
%    &     0 = A_{21}\xi_1 + A_{22}\xi_2 + B_2u  \\
% \end{split}
 %\end{equation}
 %The first two equations of \eqref{dae:sys1} can be interpreted as follows.
 \[
   \begin{split}
    & G = \begin{bmatrix} A_{12}, & B_1 \end{bmatrix}  \mbox{, \ \ }
     \widetilde{D} = \begin{bmatrix} A_{22}, & B_2 \end{bmatrix}  \mbox{ and }
     \widetilde{C}=A_{21}
   \end{split}.
 \]
 %Then \eqref{dae:sys1} is equivalent to
 %\begin{equation}
 %\label{dae:sys2}
  %\begin{split}
  %   & \dot p = \widetilde{A}p+Gq  \widetilde{C}p + \widetilde{D}q
  %  \end{split}
 %\end{equation}
  Consider the following linear system
  %with the state $p$, input $q$ and output $z$
 \begin{equation}
 \label{dae:sys_lin}
    \LIN \left\{ \begin{split}
     & \dot p = \widetilde{A}p+Gq  \\
     &  z     = \widetilde{C}p + \widetilde{D}q
    \end{split}\right..
 \end{equation}
 The trajectories $(x,u)$ of the DAE \eqref{dae:sys} are exactly those
 trajectories $(p,q)$, $T^{-1}x=(p^T,q_1^T)^T$, $q=(q_1^T,u^T)^T$, $q_1 \in \mathbb{R}^{n-r}$,
  of the linear system
 \eqref{dae:sys_lin} for which the output $z$ is zero.

 Recall from \cite[Section 7.3]{TrentelmanBook} the problem of making the output zero
 by choosing a suitable input.
 Recall from \cite[Definition 7.8]{TrentelmanBook}
 the concept
 of a weakly observable subspace of a linear system. If we apply this concept to
 $\LIN$, then an initial state $p(0) \in \mathbb{R}^{r}$ of
 $\LIN$ is \emph{weakly observable}, if there exists an input function $q \in L^2([0,+\infty), \mathbb{R}^k)$
  such that the resulting output function $z$ of $\LIN(\Sigma)$ equal zero, i.e. $z(t)=0$ for
 all $t \in [0,+\infty)$.
 Following the convention of \cite{TrentelmanBook}, let us denote the set of all
 weakly observable initial states by $\V(\LIN)$. As it was remarked in
 \cite[Section 7.3]{TrentelmanBook}, $\V(\LIN)$ is a vector space and in fact
 it can be computed.   Moreover, if $p(0)$ in $\V(\LIN)$ and for the
 particular choice of $q$, $z=0$, then $p(t) \in \V(\LIN)$ for all
 $t \ge 0$.

 Let $I=[0,t]$ or $I=[0,+\infty)$.
 Let $q \in L^2(I,\mathbb{R}^{n-r+m})$ and let $p_0 \in \mathbb{R}^{r}$.
 Denote by $p(p_0,q)$ and $z(p_0,q)$ the state and output trajectory of
 \eqref{dae:sys_lin} which corresponds to the initial state $p_0$ and input $q$.
 For technical purposes we will need the following easy extension of \cite[Theorem 7.10--.11]{TrentelmanBook}.
\begin{Theorem} %{\cite[Theorem 7.10--.11]{TrentelmanBook}}
\label{theo:geo_contr}
\begin{enumerate}
%\item \highlight{Empty space????}
 %\label{theo:geo_contr:part1}
 %     $\V=\V(\LIN)$ is the largest subspace of $\mathbb{R}^{r}$
 %     for which it holds that
 %     $\begin{bmatrix} \widetilde{A} \\ \widetilde{C} \end{bmatrix} \V \subseteq \V \times 0 + \IM \begin{bmatrix} G \\ \widetilde{D} \end{bmatrix}$.
\item
 \label{theo:geo_contr:part2}
$\V=\V(\LIN)$ is the largest subspace of $\mathbb{R}^{r}$ for which
      there exists a linear map
      $\widetilde{F}:\mathbb{R}^r \rightarrow \mathbb{R}^{m+n-r}$ such that
      \begin{equation}
      \label{theo:geo_contr:eq1}
        (\widetilde{A}+G\widetilde{F})\V \subseteq \V \mbox{ and }
        (\widetilde{C}+\widetilde{D}\widetilde{F})\V = 0
      \end{equation}
\item
 \label{theo:geo_contr:part3}
     Let $\widetilde{F}$ be a map such that \eqref{theo:geo_contr:eq1} holds for
      $\V=\V(\LIN)$. Let
      $L \in \mathbb{R}^{(m+n-r) \times k}$ for some $k$
      be a matrix such that $\IM L = \ker \widetilde{D} \cap G^{-1}(\V(\LIN))$
      and $\Rank L=k$.

      For any interval $I=[0,t]$ or $I=[0,+\infty)$,
      and for any $p_0 \in \mathbb{R}^r$, $q \in L^2_{loc}(I,\mathbb{R}^{k})$,
      \[ z(p_0,q)(t)=0  \mbox{ for } t \in I \mbox{ a.e.} \]
       if and only if $p_0 \in \V$ and there exists $w \in L^2_{loc}(I,\mathbb{R}^{n-r+m})$
       such that
      \[ q(t)=\widetilde{F}p(p_0,q)(t)+Lw(t) \mbox{ for } t \in I \mbox{ a.e.} \]
\end{enumerate}
\end{Theorem}
%\begin{proof}[Proof of Theorem \ref{theo:geo_contr}]
% Part \ref{theo:geo_contr:part1} -- \ref{theo:geo_contr:part3} are
% reformulations of \cite[Theorem 7.10]{TrentelmanBook}.
% For $I=[0,+\infty)$, Part \ref{theo:geo_contr:part3} is a restatement of
% \cite[Theorem 7.11]{TrentelmanBook}. For $I=[0,t_1]$, the proof is simillar to
% \cite[Theorem 7.11]{TrentelmanBook}.
 %%Assertions of the part \ref{theo:geo_contr:part3} could be obtained using \cite[Proposition 3]{Zhuk2012sysid}.
%\end{proof}
 We are ready now to finalize the proof of Theorem \ref{dae2lin:theo}.
 The desired linear system $\mathscr{S}=(A_l,B_l,C_l,D_l)$ is now obtained as follows.
 %First, we apply the
 %the feedback transformation $w=Ip+Jv$, where $I=-(\mathscr{L}^T\mathscr{L})^{-1}\mathscr{L}^T\mathscr{F}|_{\V}$ and $J=-(\mathscr{L}^T\mathscr{L})^{-1/2}$, to the restriction to $\V$ of the linear system below.
 Consider the linear system below.
\[
 \begin{split}
  & \dot p =(\widetilde{A}+G\widetilde{F})p+GLw \\
 & (x^T,u^T)^T = \bar{C}p+\bar{D}w \\
 & \bar{C} = \begin{bmatrix} T & 0 \\ 0 & I_m \end{bmatrix} \begin{bmatrix} I_r \\ \widetilde{F} \end{bmatrix} \mbox{ and }
  \bar{D} = \begin{bmatrix} T & 0 \\ 0 & I_m \end{bmatrix}\begin{bmatrix} 0 \\ L \end{bmatrix}.
 \end{split}
\]
%Consider the decomposition $\bar{C}_{s}:\V \rightarrow \mathbb{R}^n$,
%$\bar{C}_{inp}:\V \rightarrow \mathbb{R}^m$,
%$\bar{D}_{s}:\mathbb{R}^{k} \rightarrow \mathbb{R}^n$
%$\bar{D}_{inp}:\mathbb{R}^{k} \rightarrow \mathbb{R}^m$
%such that $\bar{C}(p)=\begin{bmatrix} \bar{C}_s(p)^T, & \bar{C}^T_{inp}(p)\end{bmatrix}^T$
%and $\bar{D}(p)=\begin{bmatrix} \bar{D}_s(v)^T & \bar{D}^T_{inp}(v)\end{bmatrix}^T$.
%It is then easy to see that $Ex_0 \in \mathcal{X}$ if and only
%if $T^{-1}x_0=(p^T,0)^T$ for some $p \in \V$.
 %Moreover,  it can be verified that $\IM E\bar{C}_s=\mathcal{X}$ and
%that $\begin{bmatrix} I_r & 0 \end{bmatrix} S$ is the left inverse of $E\bar{C}_s$ and thus $\Rank E\bar{C}_s=\hat{n}$.
 %Furthermore, it can easily be checked that $E\bar{D}_s=0$.
 %$\bar{C}^T\bar{D}=0$ and $\bar{D}^T\bar{D}=I_k$.
 Choose a basis of $\V=\V(\LIN)$ and choose $(A_l,B_l,C_l,D_l)$ as follows:
 $D_l=\bar{D}$, and let $A_l,B_l,C_l$ be the matrix representations in this basis of
 the linear maps $(\widetilde{A}+G\widetilde{F}):\V \rightarrow \V$, $GL:\mathbb{R}^k \rightarrow \V$, and
 $\bar{C}:\V \rightarrow \mathbb{R}^{n+m}$ respectively.
 Define
 \[ \mathcal{X}=\{ S^{-1}\begin{bmatrix} p \\ 0 \end{bmatrix} \mid p \in \V\}. \]
 It is easy to see that this choice of $(A_l,B_l,C_l,D_l)$ and $\mathcal{X}$
 satisfies the conditions of the theorem.

\end{proof}
\begin{Remark}[Regular case]
 The well-known case when \eqref{dae:sys} is regular, i.e. when
 $\det(sE-\hat{A}) \ne 0$ has the following interpretation.
 In this case the linear system $\LIN$ from the proof of
 Theorem \ref{dae2lin:theo} is left invertible, and $\V(\LIN)=\mathbb{R}^r$.
%% Consider the system matrix of $\LIN(\Sigma)$
%% \(
%%        P(s)=\begin{bmatrix}
%%                            A_1 -sI & G \\
%%                            C_1 & D
%%                           \end{bmatrix}
%%  \). It is easy to see that
%%\(
%%   P(s)= S
%%                       \begin{bmatrix}
%%                       \hat{A}-sE & B
%%                       \end{bmatrix} \begin{bmatrix} T & 0 \\ 0 & I_m \end{bmatrix}
%%\)
%%Hence, $\det(sE-\hat{A}) \ne 0$ is equivalent to
%%$\Rank P(s) = n$ except finitely many $s \in \mathbb{C}$.
%%From \cite{TrentelmanBook}, the latter implies that $\LIN(\Sigma)$ is left-invertable.
\end{Remark}
%\begin{Remark}[Uniqueness of the linear system associated with DAE]
%\label{rem:uniq}
 The proof of Theorem \ref{dae2lin:theo} is constructive and yields
 an algorithm for computing $(A_l,B_l,C_l,D_l)$ from $(E,\hat{A},\hat{B})$.
 This prompts us to introduce the following terminology.
\begin{Definition}
\label{linassc:def}
 A linear system $\mathscr{S}=(A_l,B_l,C_l,D_l)$ described in the proof of
 Theorem \ref{dae2lin:theo} is called
 the linear system associated with the DAE \eqref{dae:sys}.
\end{Definition}
 Note that the linear system associated  with $(E,\hat{A},\hat{B})$ is not unique.
 There are two sources of non-uniqueness:
 \begin{enumerate}
 \item The choice of the matrices $S$ and $T$ in \eqref{dae:sys:tr}.
 % $SET=\begin{bmatrix} I_r & 0 \\ 0 & 0 \end{bmatrix}$.
 \item The choice of $\widetilde{F}$ and $L$ in Theorem \ref{theo:geo_contr}.
 \end{enumerate}
 However, we can show that all associated linear systems are feedback equivalent.
 \begin{Definition}[Feedback equivalence]
  Two linear systems $\mathscr{S}_i=(A_i,B_i,C_i,D_i)$, $i=1,2$ and
  are said to be
  \emph{feedback equivalent}, if there exist a linear state feedback
  matrix $K$ and a non-singular square matrix $U$ such that
  $(A_1+B_1K,B_1U,C_1+D_1K,D_1U)$ and $\mathscr{S}_2$ are algebraically simillar.
  %$S_1$ by applying a linear state feedback $F$, a linear change of coordinates $U$ on
  %the input space and a linear change of coordinates $T$ on the state-space.
   %there exist square nonsingular matrices $T$, $U$ and a matrix
   %$F$ such that
  %\begin{align*}
  %     & TA_1T^{-1}=A_2+B_2F \mbox{ and } C_1T^{-1} = C_2+D_2F \\
  %     & TB_1U^{-1}=B_2 \mbox{ and } D_1U^{-1}=D_2
  %\end{align*}
 %We will call $(T,F,U)$ feedback equivalence
 %from $\mathscr{S}_1$ to $\mathscr{S}_2$.
 \end{Definition}
 %It is not difficult to see that state feedback equivalence is an equivalence relation.
 %If $S_1$ and $S_2$ are feedback equivalent, and
 %$y$ and $x$ is the output  and state trajectory of $S_1$ corresponding to an input $v$,
 %then $y$ and $Tx$ is the output and state trajectory $S_2$ which corresponds to input
 %$w=FTx+Uv$.
  %Moreover, all output and state trajectories of $S_2$ arise in such a way from
 %output and state trajectories of $S_1$. In particular, the sets of all output trajectories of
 %$S_1$ and $S_2$ coincide.
%% \begin{Definition}[State-input isomorphism]
%%  We will say that two linear systems $\mathscr{S}_i=(A_i,B_i,C_i,D_i)$, $i=1,2$
%%  are state-input isomorphic,
%%  % if the input-, state- and ouput-spaces of
%%  %$\mathscr{S}_1$ and $\mathscr{S}_2$ coincide, and
%%  there exist nonsingular square matrices
%%  $T,U$, such that $U$ is orthogonal and
%%  $TA_1T^{-1}=A_2$, $TB_1=B_2U$, $C_1T^{-1}=C_2$ and $D_1=D_2U$.
%%  The pair $(T,U)$ will be called the state-input isomorphism from
%%  $\mathscr{S}_1$ to $\mathscr{S}_2$.
%% \end{Definition}
%%  Intuitively,
%%  if $y$ and $x$ are the output and state trajectory of $\mathscr{S}_1$ which
%%  correspond to the input $v$, then $y$ and $Tx$ are the output and state trajectory of
%%  $\mathscr{S}_2$ which correspond to the input $Uv$.
 \begin{Lemma}
 \label{lemm:uniq}
  Let $\mathscr{S}_i=(A_i,B_i,C_i,D_i)$, $i=1,2$ be
  two linear systems which are obtained from the proof of Theorem \ref{dae2lin:theo}.
%  and let $\LMAP_i$, $i=1,2$
%  be the corresponding state maps.
 Then $\mathscr{S}_1$ and $\mathscr{S}_2$ are feedback equivalent.
 % Moreover, if $(T,U)$ is the corresponding
 %state-input isomorphism, from $\mathscr{S}_1$ to $\mathscr{S}_2$,
 %then $T\LMAP_1=\LMAP_2$.
 \end{Lemma}
 The proof of Lemma \ref{lemm:uniq} can be found in the appendix.

\subsection{Solution of the optimal control problem for DAE}
\label{dae:opt}
We apply Theorem~\ref{dae2lin:theo} in order to solve a control problem defined in
 Problem \ref{opt:contr:def}.
%  First, notice that in the control problem above, we can assume without loss of
%  generality that $R=I_m$ and $Q=I_n$. Indeed, if this not the case, then
%  we can perform the coordinate transformation $\hat{x}=Q^{1/2}x$ and
%  $\hat{u}=R^{1/2}u$. If we now define $\widetilde{E}=EQ^{-1/2}$, $\widetilde{A}=\hat{A}Q^{-1/2}$,
%  $\widetilde{B}=\hat{B}R^{-1/2}$, and $\widetilde{Q}_0=Q^{-1/2}Q_0Q^{1/2}$. It then follows
%  that the control problem $\mathcal{C}(E,\hat{A},\hat{B},Q,R,Q_0)$,  and
%  %$\mathcal{C}(E,\hat{A},\hat{B},Q,R,Q_0)$
%  is equivalent to
%  $\mathcal{C}(\widetilde{E},\widetilde{A},\widetilde{B},I_n,I_m,\widetilde{Q}_0)$:
%  the tuple $(x^{*},u^{*})$ is a solution to
%  $\mathcal{C}(E,\hat{A},\hat{B},Q,R,Q_0)$ if and only if
%  $(Q^{1/2}x^{*},R^{1/2}u^{*})$ is a solution to
%  $\mathcal{C}(\widetilde{E},\widetilde{A},\widetilde{B},I_n,I_m,\widetilde{Q}_0)$.
% \emph{In the rest of this section, we will assume that $R=I_m$ and $Q=I_n$}.
  %We continue by translating \eqref{opt1} to a classical linear quadratic
  %optimal control problem.
  Let $\mathscr{S} = (A_l,B_l,C_l,D_l)$ be a linear system associated with
  $\Sigma$ and let $\LMAP$ be the map described in Theorem \ref{dae2lin:theo} and let
  $C_s$ the component of $C_l$ as defined in Theorem \ref{dae2lin:theo}.
  Consider the following linear quadratic control problem.
  %Denote by $\mathcal{D}_{\mathscr{S},v_0}(t_1)$, $t_1 \in [0,+\infty]$ the set of
  %all trajectories $(v,g) \in A(I,\mathbb{R}^{\hat{n}}) \times L^2_{loc}(I,\mathbb{R}^{k})$,
  %$I=[0,t_1] \cap [0,+\infty)$, such that
  %$\dot v = Av + Bg$, $v(0)=v_0$.
  For every initial state $v_0$, for every interval $I$ containing $[0,t_1]$ and for every
  $g \in L^2_{loc}(I,\mathbb{R}^{k})$ define the cost functional $J(v_0,g,t)$
  \[
    \begin{split}
     & \mathscr{J}(v_0,g,t_1) =  v^T(t_1)E^TC_s^T Q_0 EC_{s} v(t_1) + \\
     & +\int_0^{t_1} \nu^T(t)\begin{bmatrix} Q & 0 \\ 0 & R \end{bmatrix} \nu(t)dt    \\
     & \dot v = A_lv + B_l g  \mbox{ and }  v(0)=v_0 \\
     & \nu = C_lv+D_lg .
    \end{split}
   \]
 For  any $g \in L^2_{loc}([0,+\infty), \mathbb{R}^{k})$ and $v_0 \in \mathbb{R}^{\hat{n}}$, define
 \[ \mathscr{J}(v_0,g)=\limsup_{t_1 \rightarrow \infty} \mathscr{J}(v_0,g,t_1). \]
 Consider the control problem of finding for every initial state $v_0$ an input
  $g^{*} \in L^2_{loc}(\mathbb{R}^k)$ such that
 \begin{equation}
 \label{opt_final}
 \begin{split}
  &  \mathscr{J}(v_0,g^{*}) = \limsup_{t_1 \rightarrow \infty} \inf_{g \in L^2(0,t_1)} \mathscr{J}(v_0,g,t_1).
 \end{split}
 \end{equation}
%% %where $\hat{S}$ and $\hat{S}_0$, $l_1$ are defined as follows.
 \begin{Definition}[Associated LQ problem]
 \label{def:class:lq}
  The control problem \eqref{opt_final} is called an \emph{LQ problem
  associated} with  $\mathcal{C}(E,\hat{A},\hat{B},Q,R,Q_0)$
  and it is denoted by
  $\mathcal{CL}(A_l,B_l,C_l,D_l)$.
 \end{Definition}
 \begin{Remark}[Uniqueness]
  %Since the linear system and the state map
  %associated with the DAE is not unique, the
  %LQ problem associated with $\mathcal{C}(E,\hat{A},\hat{B},Q,R,Q_0)$
  %is not unique either.
  Note the solution of an associated LQ does not depend on the choice of
  $\mathscr{S}$: for any two choices of $\mathscr{S}$, the corresponding
  solutions can be transformed to each other by a linear state feedback
  and linear coordinate changes of the input- and state-space.
  %$\mathscr{S}_i$,
  %$i=1,2$ are two linear systems associated with the DAE, then they are
  %feedback equivalent. In turn, feedback equivalence preserves solutions of
  %the optimal control problem \eqref{opt_final}.
  %then $\mathscr{S}_1$ and $\mathscr{S}_2$ are related by a
  %state-input isomorphism $(T,U)$.
  %Thus, if $v^{*}$ is a solution to
  %$\mathcal{CL}(A_1,B_1,C_1,D_1)$
  %$g_2^{*}=Ug^{*}$ is a solution of
 %$\mathcal{CL}(A_2,B_2,C_2,D_2)$.
  %and the latter control
  %problem corresponds to the linear control problem which arises by
  %choosing $(A_2,B_2,C_2,D_2)$ instead of $(A_1,B_1,C_1,D_1)$.
 \end{Remark}
  The relationship between the associated LQ problem and the
  original control problem for DAEs is as follows.
 \begin{Theorem}
 \label{lqg:eq}
 %Consider an optimal DAE control problem
 %$\mathcal{C}(E,\hat{A},\hat{B},I_n,I_m,Q_0)$,
  Let $g^{*} \in L^2_{loc}([0,+\infty), \mathbb{R}^{k})$ and let $(x^{*},u^{*})$
  be the corresponding output of $\mathscr{S}=(A_l,B_l,C_l,D_l)$ from the initial state
  $v_0 =\LMAP(Ex_0)$ for some $x_0 \in \mathbb{R}^n$, $\mathscr{D}_{x_0}(\infty) \ne \emptyset$.
   Then $(x^{*},u^{*}) \in \mathscr{D}_{x_0}(\infty)$ and 
   $g^{*}$ is a solution of $\mathcal{CL}(A_l,B_l,C_l,D_l)$ for $v_0$  if and only if
  \begin{equation*}
  \label{lqg:eq1} 
    J(x^{*},u^{*})=\limsup_{t_1 \rightarrow \infty} \inf_{(x,u) \in \mathscr{D}_{x_0}(t_1)} J(x,u,t). 
  \end{equation*}
  %In particular, $\mathcal{CL}(A_l,B_l,C_l,D_l)$ has a solution from $\LMAP(Ex_0)$ 
  %if and only if
  %there exists $(x^{*},u^{*}) \in \mathscr{D}_{x_0}(t_1)$ which satisfies
  %\eqref{lqg:eq1}.
 \end{Theorem}
 \begin{proof}[Proof of Theorem \ref{lqg:eq}]
  Assume that $I=[0,t_1] \cap [0,+\infty)$, $t_1 \in [0,+\infty]$.
  The theorem follows by noticing that for any $g \in L^2_{loc}(I,\mathbb{R}^k)$,
  the output $(x,u)$ of $\mathscr{S}$ from $v_0=\LMAP(Ex_0)$ has the property
  that $(x,u) \in \mathcal{D}_{x_0}(t_1)$, and if
  $t_1 < +\infty$, then
  $J(x,u,t_1)=\mathscr{J}(\LMAP(Ex_0),g,t_1)$ and if $I=[0,+\infty)$, then
  $J(x,u)=\mathscr{J}(\LMAP(Ex_0),g)$. Moreover, any element of
  $\mathcal{D}_{x_0}(t_1)$ arises as an output of $\mathscr{S}$ for some $g \in L^2_{loc}(I,\mathbb{R}^k)$.
 \end{proof}
 The solution of associated LQ problem can be derived using
  classical results, see \cite{KwakernaakBook}.
 \begin{Theorem}%[Optimal control]
 \label{opt:control}
 Let $\mathcal{CL}(A_l,B_l,C_l,D_l)$ be the LQ problem associated with
 $\mathcal{C}(E,\hat{A},\hat{B},Q,R,Q_0)$.
 %Then $\mathcal{CL}(A_l,B_l,C_l,D_l)$ has a solution if and only if 
 Assume that $(A_l,B_l)$ is stabilizable. 
  Define $S=\begin{bmatrix} Q & 0 \\ 0 & R \end{bmatrix}$.
    Consider the algebraic Riccati equation
    \begin{equation}
    \label{ARE}
    \begin{split}
       & 0 = PA_l+A^T_lP- K^T(D^T_lSD_l)K + C^T_lSC_l. \\
       & K=(D^T_lSD_l)^{-1}(B^T_lP+D^T_lSC_l).
    \end{split}
    \end{equation}
  %\item
    Then \eqref{ARE} has a unique solution $P > 0$, and 
    $A_l- B_lK$ is a stable matrix. Moreover, if $g^{*}$ is
    defined as
   \begin{equation}
   \label{opt8}
    \begin{split}
      & \dot v^{*} = A_lv^{*}+B_lg^{*} \mbox{ and } v^{*}(0)=v_0 \\
      & g^{*} = - Kv^{*}, \\
   \end{split}
   \end{equation}
    then $g^{*}$ is a solution of $\mathcal{CL}(A_l,B_l,C_l,D_l)$ for the initial
    state $v_0$ and
    $v_0^TPv_0=\mathscr{J}(v_0,g^{*})$.
  %\item
    %The optimal solution $g^{*}$ satifies the following
  %  equalities
  %  \begin{equation}
  %  \label{opt8.1}
  % \begin{split}
  %   & J(g^{*})=\lim_{t_1\rightarrow \infty} \mathscr{J}(v_0,g^{*},t_1)  \\
  %   & J(g^{*})=\lim_{t_1 \rightarrow \infty} \inf_{g \in L_2(0,t_1)}  \mathscr{J}(v_0,g,t_1)
  %  \end{split}
  % \end{equation}
%\end{enumerate}
\end{Theorem}
\begin{proof}[Proof of Theorem \ref{opt:control}]
 %The first statement is a direct consequence of
 %\cite[Theorem 3.7]{KwakernaakBook}.
   Let us first apply the feedback transformation $g=\hat{F}v+Uw$ to $\mathscr{S}=(A_l,B_l,C_l,D_l)$
  with $U=-(D^T_lSD_l)^{-1/2}$ and $\hat{F}=-(D^T_lSD_l)^{-1}D^T_lSC_l$. Consider the linear system
  \begin{equation}
  \label{opt:st1}
   \begin{split}
   & \dot v = (A_l+B_l\hat{F})v+B_lUw \mbox{ and } v(0)=v_0 \\
   %& z=S^{1/2}(C_l+D_l\hat{F})v+D_lUw.
   \end{split}
  \end{equation}
  For any $w \in L^2_{loc}(I)$, where $I=[0,t_1]$ of $I=[0,+\infty)$,
  the state trajectory $v$ of \eqref{opt:st1} equals the state trajectory of
  $\mathscr{S}$ for the input $g=\hat{F}v+Uw$ and initial state $v_0$.
  Moreover, all inputs $g$ of $\mathscr{S}$ can be represented in such a way.
  Define now
   \begin{equation*}
  %\label{opt:st2}
  \begin{split}
      & \widehat{\mathscr{J}}(v_0,w,t)= v^T(t)E^TC_s^T\bar{Q}_0 EC_{s} v(t)+ \\
       & +\int_0^t (v^T(t)(C_l+D_l\hat{F})^TS(C_l+D_l\hat{F})v(t)+w^T(t)w(t))dt,
   \end{split}
  \end{equation*}
  where $v$ is a solution of \eqref{opt:st1}.
  %For $I=[0,+\infty)$, define
  %$\widehat{\mathscr{J}}(v_0,w)=\limsup_{t \rightarrow \infty} \widehat{\mathscr{J}}(v_0,w,t)$.
  It is easy to see that for $g=\hat{F}v+Uw$, $\mathscr{J}(v_0,g,t)=\widehat{\mathscr{J}}(v_0,w,t)$.

   Consider now the problem of
   minimizing $\lim_{t \rightarrow \infty} \widehat{\mathscr{J}}(v_0,w,t)$.
   The solution of this problem can be found using \cite[Theorem 3.7]{KwakernaakBook}.
   To this end, notice
   that $(A_l+B_l\hat{F},B_lU)$ is stabilizable and $(S^{1/2}(C_l+D_l\hat{F}),A_l+B_l\hat{F})$ is observable.
   Indeed, it is easy to see that stabilizability of $(A_l,B_l)$ implies that of $(A_l+B_l\hat{F},B_lU)$.
   Observability  $(S^{1/2}(C_l+D_l\hat{F}),A_l+B_l\hat{F})$ is implied by the fact that
   by
  Theorem \ref{dae2lin:theo}, $EC_s$ is full column rank and $ED_s=0$, and thus
   $E(C_s+D_s\hat{F})=EC_s$ is full column rank.
   Furthermore, notice that \eqref{ARE} is equivalent to the algebraic Riccati equation
   described in \cite[Theorem 3.7]{KwakernaakBook} for the problem of minimizing
   $\lim_{t \rightarrow \infty} \widehat{\mathscr{J}}(v_0,w,t)$.
  %\begin{equation}
  %\label{opt:st4}
  %  (A+BF)^TP+P(A+BF)-PBUU^TBP+(C+DF)^TS(C+DF) = 0.
  %\end{equation}
 Hence, by \cite[Theorem 3.7]{KwakernaakBook},
 \eqref{ARE}  has a unique positive definite solution $P$, 
  and $A_l-B_l(\hat{F}+U^TB_lP)=A_l-B_lK$ is a stable matrix. %Recall that $K=(D^TSD)^{-1}(B^TP+D^TSC)$.
   From \cite[Theorem 3.7]{KwakernaakBook}, there exists $w^{*}$ such that
   $\lim_{t \rightarrow \infty} \widehat{\mathscr{J}}(v_0,w^{*},t)$ is minimal.
  and $v_0^TPv_0 = \lim_{t \rightarrow \infty} \widehat{\mathscr{J}}(v_0,w^{*},t)$.
  From \cite[Theorem 3.7]{KwakernaakBook} we can also deduce that
  %Moreover, if $P(t)$ is the solution of
  %the differential Riccati equation which corresponds to the control problem of
  %minimizing $\widehat{\mathscr{J}}(v_0,w,t)$, then
  %\begin{equation*}
  %\label{RE}
 %   \begin{split}
    %& \dot P(t) = P(t)(A+BF)+(A+BF)^TP(t) -P(t) BUU^TB^T P(t) + (C+DF)^TS(C+DF) \\
    %& P(0)=E^TC_s^T\hat{Q}_0EC_s,
    %\end{split}
  %\end{equation*}
   %then  %by \cite[Theorem 3.7]{KwakernaakBook},
  %\( v_0^TP(t)v_0=\inf_{w \in L^2(0,t)} \widehat{\mathscr{J}}(v_0,w,t), \) and
  %and  $\lim_{t \rightarrow \infty} P(t)=P$.
  %Hence, we conclude that
  $v_0^TPv_0=\lim_{t_1 \rightarrow \infty} \inf_{w \in L^2(0,t_1)} \widehat{\mathscr{J}}(v_0,w,t_1)$.

  Hence, $g^{*}=\hat{F}v^{*}+Uw^{*}$ is a solution of $\mathcal{CL}(A_l,B_l,C_l,D_l)$ for the initial state
  $v_0$, where $v^{*}$ is the solution of \eqref{opt:st1} which corresponds to $w=w^{*}$.
  A routine computation reveals that $(v^{*},g^{*})$ satisfies \eqref{opt8}.

 %Notice that $S(t)=P$ satisfies the Ricatti equation
 %$-\dot S(t) + A^TS(t)+S(t)A - S(t)BB^TS(t)+C^TC=0$ with $S(0)=P$. Hence,
 %$\mathscr{J}(g^{*},t)=v_0^TPv_0 + (v^{*}(t))^T\hat{S}v^{*}(t)-(v^{*}(t))^TPv^{*}(t)$.
% Since $P(0)=\hat{Q}_0$ and $P(t)$ is monotonically increasing, it follows that
% $(v^{*}(t))^T\hat{S}v^{*}(t)-(v^{*}(t))^TPv^{*}(t) \le 0$ and hence
% $\mathscr{J}(v_0,g^{*},t) \le v_0^TPv_0$.
 %From $\lim_{t \rightarrow \infty} v(t)=0$
 %it follows that $\lim_{t \rightarrow \infty} \mathscr{J}(v_0,g^{*},t)=v_0^TPv_0$.
% Hence,
% $\limsup_{t} J(v_0,g^{*},t)=\lim_{t \rightarrow \infty} J(v_0,g^{*},t)=v_0^TPv_0 < +\infty$.
\end{proof}
Combining Theorem \ref{opt:control} and Theorem \ref{lqg:eq}, we
can solve
the optimal control problem for DAEs as follows.
%%the solution of
%% \begin{Corollary}
%% \label{col1}
%% Choose an associated linear
%% system $(A,B,C,D)$ and a state-map $\LMAPINV$.
%% Let $P(t)$ be the solution of the Riccati equation \eqref{RE} which
%% corresponds to the control problem
%% $\mathcal{CL}(A,B,C,D,\LMAPINV,t_1)$. be the corresponding linear optimal control problem.
%% Then ${\LMAPINV}^T(Ex_0)P(t_1)\LMAPINV(Ex_0)=J(u^{*},x^{*},d^{*},t_1)$,
%% $d^{*}=\LOPT(Ex^{*}(t_1))$ and for all $t \in [0,t_1]$,
%% \[
%%    \begin{split}
%%    & \frac{d(Ex^{*}(t))}{dt} = \hat{A}x^{*}(t)+\hat{B}u^{*}(t) \mbox{ and } Ex^{*}(0)=Ex_0 \\
%%   & u^{*}(t)=\hat{K}(t)x^{*}(t) \\
%%    & \hat{K}(t)=\begin{bmatrix} 0 & 0 \\ 0 & I_{m} \end{bmatrix}(C-DB^TP(t))\LMAPINV E
%%    \end{split}
%%  \]
%% \end{Corollary}
%% Note that the solution presented in Corollary \ref{col1} is essentially
%% equivalent to the one in \cite{Zhuk2012sysid}
% For the infinite horizon case, we can proceed in a simillar way.
 \begin{Corollary}
 \label{col2}
 Consider the control problem
 $\mathcal{C}(E,\hat{A},\hat{B},Q,R,Q_0)$ and let
 $\mathcal{CL}(A_l,B_l,C_l,D_l)$ be an LQ problem associated with $\mathcal{C}(E,\hat{A},\hat{B},Q,R,Q_0)$.
 Assume that $(A_l,B_l)$ is stabilizable.
 Let $P$ be the unique positive definite solution of \eqref{ARE} and
 let $K$ be as in \eqref{ARE}.
 Let $C_s,C_{inp}$, $D_s$, $D_{inp}$ be the decomposition of $C_l$ and $D_l$ as
 defined in Theorem \ref{dae2lin:theo} and let $\LMAP=(EC_s)^{+}$.
 Then the dynamical controller $\mathscr{C}=(A_c,B_{c},C_x,C_u)$ with
 \begin{align*}
   & A_c=A_l-B_lK \mbox{, } C_x=C_s-D_sK \\
   & C_u=(C_{inp}-D_{inp}K) \mbox{ and }  B_c=\LMAP.
 \end{align*}
 is a solution of $\mathcal{C}(E,\hat{A},\hat{B},Q,R,Q_0)$.
%% Then a solution $(x^{*},u^{*},d^{*})$ of $\mathcal{C}(E,\hat{A},\hat{B},I_n,I_m,Q_0,x_0,\infty)$ can be obtained as follows:
%% $d^{*}=\LOPT(Ex^{*})$, and
%% \[
%%   \LMAPINV(Ex_0)^TP \LMAPINV(Ex_0) = J(x,u,d)
%% \]
%% and for all $t \in [0,+\infty)$,
%% \begin{align*}
%%    & \frac{d(Ex^{*}(t))}{dt} = \hat{A}x^{*}(t)+\hat{B}u^{*}(t) \mbox{ and } Ex^{*}(0)=Ex_0 \\
%%    & u^{*}(t)=\hat{K}x^{*}(t) \\
%%    & \hat{K}=\begin{bmatrix} 0 & 0 \\ 0 & I_{m} \end{bmatrix}(C-DB^TP)\LMAPINV E
   %\begin{split}
   %\dot v^{*} = (A-BB^TP)v^{*} \mbox{ and } v^{*}=v_0 \\
   %(x^{*},u^{*}) =
 %%\begin{bmatrix} Q^{-1/2} & 0 \\ 0 & R^{-1/2} end{bmatrix}
 %(C-DB^TP)v^{*}
   %\end{split}
 %\end{align*}
 %Moreover, $A-BB^TP$ is a stable matrix.
\end{Corollary}
 \begin{Remark}[Computation and existence of a solution]
  The existence of solution for Problem \ref{opt:contr:def} and its computation
  depend only on the matrices $(E,\hat{A},\hat{B},Q,R,Q_0)$.
  Indeed, a linear system $\mathscr{S}$ associated with $(E,\hat{A},\hat{B})$
  can be computed from $(E,\hat{A},\hat{B})$, and the solution of the
  associated LQ problem can be computed using $\mathscr{S}$ and the
  matrices $Q,Q_0,R$.
  Notice that the only condition for the existence of a solution is that
  $\mathscr{S}=(A_l,B_l,C_l,D_l)$ is stabilizable. Since all linear systems associated with
  the given DAE are feedback equivalent, stabilizability of an associated
  linear system does not depend on the choice of the linear system.
  Thus, stabilizability of $\mathscr{S}$ can be regarded as a property of $(E,\hat{A},\hat{B})$. The link between stabilizability of $\mathscr{S}$ and the
  classical stabilizability for DAEs remains a topic for future research.
 \end{Remark}

\subsection{Observer design for DAE}
\label{sec:obs}
 By applying Corollary \ref{col2} and Theorem \ref{p:5}, we obtain the following
 procedure for solving Problem \ref{problem:obs}.
 \begin{itemize}
 \item{\textbf{Step 1.}}
     Consider the dual DAE of the form \eqref{dae:sys}, such that
     $F^T=E$, $A^T=\hat{A}$ and $-H^T=\hat{B}$.
      %\begin{equation}
      %\label{sec:obs:eq1}
      %    \dfrac{F^Tz}{dt} = A^Tz - H^Tu.
      %\end{equation}
    Construct a linear system $\mathscr{S}=(A_l,B_l,C_l,D_l)$ associated
    with this DAE, as described in Definition \ref{linassc:def}.

 \item{\textbf{Step 2.}}
    Check if $(A_l,B_l)$ is stabilizable. If it is, let
    \[ X=\begin{bmatrix} Q^{-1} & 0 \\ 0 & R^{-1} \end{bmatrix}. \]
    Consider the algebraic Riccati equation
    \begin{equation}
     \label{sec:obs:eq2}
    \begin{split}
       & 0 = PA_l+A^T_lP- K^T(D^T_lXD_l)K + C^T_lXC_l. \\
       & K=(D^T_lXD_l)^{-1}(B^T_lP+D^T_lXC_l).
    \end{split}
    \end{equation}
    The equation \eqref{sec:obs:eq2} has a unique solution $P > 0$.
 \item{\textbf{Step 3.}}
    The dynamical observer $\mathscr{O}_{\widehat{U}}$ which is a solution of Problem 
    \ref{problem:obs} is of the form:
   \[
     \begin{split}
     & \dot r(t) = (A_l-B_lK)^Tr(t) + (C_l-D_lK)^T\begin{bmatrix} 0 \\ y(t) \end{bmatrix} \\
      & \mathcal{O}_{\widehat{U}}(y)(t)=\ell^TF\LMAP^T r(t), 
     \end{split}
   \]
  and $\widehat{U}(t,s)=(C_l-D_lK)e^{(A_l-B_lK)(t-s)}\LMAP F^T\ell$.
  The observation error equals
  \[ \sigma(\widehat{U},\ell) = \ell^TF\LMAP^TP\LMAP F^T\ell. \]
  Recall that
  $\LMAP=(F^TC_s)^{+}$, where $C_s$ is the submatrix of $C_l$ formed by its first $n$ rows.
\end{itemize}
%\highlight{Could you make an exact link to DAE~\eqref{eq:dae_output}?}\textbf{It is not clear!!! You do not even give a sketch. I am OK to omit details but the clear sketch should be in place: we take DAE~\eqref{eq:dae_output}, do 1. 2. 3. and get the observer in the presented form.}
 \begin{Remark}[Conditions for existence of an observer]
 The existence of the observer above depends only on whether
 the chosen linear system associated
 with the dual DAE is stabilizable. As it was mentioned before, the latter
 is a property of the tuple $(F,A,H)$.
 %the particular choice of the linear system associated with the dual DAE does not
 %influence stabilizability and it
 %depends only on the tuple of matrices
 %$(E,A,H)$.
  Hence, the property that the linear system associated
 with the dual DAE is stabilizable
 could be thought of as a sort of detectability property.
 The relationship between this property and the detectability notions 
 established in the literature remains a topic of future research.
 %If $(E,A)$ is a regular DAE, then the detectability condition of
 %\cite{BenderLaub1987IEEETAC} implies that the associated linear system is
 %detectable.
\end{Remark}

\section{Conclusions}
 We have presented a solution to the minimax observer design problem and the infinite horizon
 linear quadratic control problem for linear DAEs. We have also shown that
 these two problems are each other's dual.  The main novelty of this contribution
 is that we made no solvability assumptions on DAEs. The only condition we need is 
 that the LTI associated with the dual DAE should be stabilizable. We conjecture
 that this condition is also a necessary one. The clarification of this issue remains
 a topic of future research.

%\end{document}

\appendix
 \begin{proof}[Proof of Lemma \ref{lemm:uniq}]
  We will use the following terminology in the sequel.
  Consider two linear systems $(A_1,B_1,C_1,D_1)$ and $(A_2,B_2,C_2,D_2)$
  with $n$ states, $p$ outputs and $m$ inputs. 
   A tuple $(T,F,G,U)$ of matrices,
    $T \in \mathbb{R}^{n \times n}$, $U \in \mathbb{R}^{m \times m}$,
    $V \in \mathbb{R}^{p \times p}$,
   $F \in \mathbb{R}^{m \times n}, G \in \mathbb{R}^{n \times p}$ such that $T,U$ and $V$ are non-singular, is said to be a \emph{feedback equivalence with output injection} from $(A_1,B_1,C_1,D_1)$ to $(A_2,B_2,C_2,D_2)$, if
  \begin{align*}
       & T(A_1+B_1F+GC_1+GD_1F)T^{-1}=A_2 \\
       & V(C_1+D_1F)T^{-1} = C_2 \\
       & T(B_1+GD_1)U=B_2 \mbox{ and } VD_1U=D_2
  \end{align*}
  If $G=0$, $V=I_p$, then $(T,F,G,U)$ is just a feedback equivalence and 
  $(A_1,B_1,C_1,D_1)$ and $(A_2,B_2,C_2,D_2)$ are feedback equivalent.
  In this case (i.e. when $G=0$, $V=I_p$), 
  we denote this transformation by $(T,F,U)$.

%% It is not difficult to see that state feedback equivalence is an equivalence relation.
%% If $S_1=(A_1,B_1,C_1,D_1)$ and $S_2=(A_2,B_2,C_2,D_2)$ are feedback equivalent, and
%% $y$ and $x$ is the output  and state trajectory of $S_1$ corresponding to an input $v$,
%% then $y$ and $Tx$ is the output and state trajectory $S_2$ which corresponds to input
%% $w=FTx+Uv$. Moreover, all output and state trajectories of $S_2$ arise in such a way from
%% output and state trajectories of $S_1$. In particular, the sets of all output trajectories of
%% $S_1$ and $S_2$ coincide.
 Let $S_i,T_i \in \mathbb{R}^{n \times n}$ 
 be invertable, such that
 $S_iET_i = \begin{bmatrix} I_r & 0 \\ 0 & 0 \end{bmatrix}$, $i=1,2$.
 Let
 \[
   \begin{split}
    & S_i\hat{A}T_i=\begin{bmatrix} A_i & A_{12,i} \\
                    A_{21,i} & A_{22,i}
     \end{bmatrix} \\
      & S_i\hat{B}=\begin{bmatrix} B_{1,i} \\ B_{2,i} \end{bmatrix}, \\
    &  G_i = \begin{bmatrix} A_{12,i}, & B_{1,i} \end{bmatrix} \\
    &  \widetilde{C}_i = A_{21,i} \mbox{ and } \widetilde{D}_i = \begin{bmatrix} A_{22,i}, & B_{2,i} \end{bmatrix}
   \end{split}
 \]
 and consider the linear systems
 \[
  \mathscr{S}_i\left\{
  \begin{split}
   & \dot{p}_i = A_ip_i+G_iq_i  \\
    & z_i = \widetilde{C}_ip_1+\widetilde{D}_iq_i
  \end{split}\right.
 \]
 for $i=1,2$.
 Denote by $\V_i=\V(\mathscr{S}_i)$ the set of weakly observable states of
 $\mathscr{S}_i$,$i=1,2$. Denote by $\mathcal{F}(\V_i)$,
 $i=1,2$, the set of all state feedback matrices $F \in \mathbb{R}^{n \times m}$ such that 
  $(A_i+G_iF)\V_i \subseteq \V_i$, $(\widetilde{C}_i+\widetilde{D}_iF)\V_i=0$. Pick $F_i  \in \mathcal{F}(\V_i)$, $i=1,2$ and pick 
 full column rank matrices $L_i$, $i=1,2$ such that 
 $\IM L_i=G_i^{-1}(\V_i) \cap \ker \widetilde{D}_i$. 
 In order to prove the lemma, it is enough to show that
 $\Rank L_1=\Rank L_2=k$, and 
 there exist invertable linear maps
 $T \in \mathbb{R}^{r \times r}$,  $U \in \mathbb{R}^{k \times k}$,
 and a matrix $F \in \mathbb{R}^{r \times k}$ such that
 \begin{subequations}
 \label{feedback:pf:eq1}
   \begin{align}
    &  T(\V_1)=\V_2 \label{feedback:pf:eq1:eq1} \\  
    & (A_1+G_1F_1+G_1L_1F)\V_1 \subseteq \V_1 
      \label{feedback:pf:eq1:eq2}
\\,
    & \forall x \in \V_1: \nonumber \\
    & T(A_1+G_1F_1+G_1L_1F)T^{-1}x=(A_2+G_2F_2)x 
      \label{feedback:pf:eq1:eq3} \\
    & TG_1L_1U=G_2L_2 
      \label{feedback:pf:eq1:eq4} \\
    &  \begin{bmatrix} T_1 & 0 \\ 0 & I_m \end{bmatrix} \begin{bmatrix} 0_{r \times k} \\ L_1U \end{bmatrix} = \begin{bmatrix} T_2 & 0 \\ 0 & I_m \end{bmatrix} \begin{bmatrix} 0_{r \times k} \\ L_2 \end{bmatrix} 
      \label{feedback:pf:eq1:eq5} \\
   & \forall x \in \V_1: \nonumber \\
   &  \begin{bmatrix} T_1 & 0 \\ 0 & I_m \end{bmatrix} \begin{bmatrix} I_r \\ (F_1+L_1F) \end{bmatrix} x= \begin{bmatrix} T_2 & 0 \\ 0 & I_m \end{bmatrix} \begin{bmatrix} I_r \\ F_2 \end{bmatrix} Tx,
      \label{feedback:pf:eq1:eq6} 
  \end{align}
 \end{subequations}
where $0_{r \times k}$ denotes the $r \times k$ matrix with all zero entries.
Indeed, the associated linear systems arising from the two
choices $T_i,S_i,F_i,L_i$, $i=1,2$ are in fact isomorphic to
the following linear system defined on $\V_i$, $i=1,2$,
\begin{equation}
 \label{feedback:pf:eq11}
  \LLIN_i\left\{\begin{split}
   & \dot p = (A_i+G_iF_i)|_{\V_i}p+G_iL_iw \\
   & (x,u)^T = 
    \begin{bmatrix} T_i & 0 \\ 0 & I_m \end{bmatrix} (F_i|_{\V_i}p+L_iw)
  \end{split}\right.
\end{equation}
If $(T,F,U)$ satisfy \eqref{feedback:pf:eq1}, it then
follows that $(T,F,U)$ is a feedback equivalence between 
$\LLIN_1$ and $\LLIN_2$.

In order to find the matrices $F,U,T$, notice that
\[ T_2^{-1}T_1=\begin{bmatrix} R_{11} & 0 \\ R_{21} & R_{22} \end{bmatrix}
\]
 for $R_{11} \in \mathbb{R}^{r \times r}$, $R_{22} \in \mathbb{R}^{(n-r) \times (n-r)}$,
 $R_{21} \in \mathbb{R}^{(n-r) \times r}$. Indeed,
 assume
 $T_2^{-1}T_1\begin{bmatrix} 0 \\ q \end{bmatrix} = \begin{bmatrix} \bar{p} \\ \bar{q} \end{bmatrix}$ for some $q,\bar{q} \in \mathbb{R}^{n-r}$, $\bar{p} \in \mathbb{R}^{r}$.
 Then
 \[ 
  \begin{split}
  & \begin{bmatrix} \bar{p} \\ 0 \end{bmatrix}  =S_2ET_2T_2^{-1}T_1 \begin{bmatrix} 0 \\ q \end{bmatrix}= \\
  & S_{2}S_1^{-1} S_1ET_1\begin{bmatrix} 0 \\ q \end{bmatrix} = S_2S_1^{-1}0=0.
  \end{split}
\]
 Hence, $T_2^{-1}T_1\begin{bmatrix} 0 \\ q \end{bmatrix} = \begin{bmatrix} 0 \\ \bar{q} \end{bmatrix}$ from which the statement follows.
 In a simillar fashion
 \[ S_2S_1^{-1}=\begin{bmatrix} H_{11} & H_{21} \\ 0 & H_{22} \end{bmatrix}, \] where
 $H_{11} \in \mathbb{R}^{r \times r}$, $H_{22} \in \mathbb{R}^{(n-r) \times (n-r)}$,
 $H_{12} \in \mathbb{R}^{r \times n-r}$, moreover,
 \[ H_{11} = R_{11}. \] 
 Indeed,
 \[ 
   \begin{split}
  &  S_2S_1^{-1}\begin{bmatrix} p \\ 0 \end{bmatrix} = S_2S_1^{-1} S_1ET_1\begin{bmatrix} p \\ 0 \end{bmatrix} = \\
  &  S_2ET_2T_2^{-1}T_1\begin{bmatrix} p \\ 0 \end{bmatrix} = \begin{bmatrix} \bar{p} \\ 0 \end{bmatrix}
   \end{split}
\] for some $\bar{p} \in \mathbb{R}^r$.
  Finally,
 \[ \begin{bmatrix} H_{11} & 0 \\ 0 & 0 \end{bmatrix} = S_2S_1^{-1} (S_1ET_1) = (S_2ET_2)T_2^{-1}T_1 = \begin{bmatrix} R_{11} & 0 \\ 0 & 0 \end{bmatrix}. \]
  Hence, 
  $R_{11}=H_{11}$.

 From $S_2\hat{A}T_2 = S_2S_1^{-1} S_1 \hat{A}T_1(T_2^{-1}T_1)^{-1}$ it and $S_2\hat{B}=S_2S_1^{-1}S_1\hat{B}$ 
 follows that
 \begin{equation}
 \label{feedback:pf:eq2}
    \begin{split}
     & A_2=R_{11}(A_1+G_1\hat{F}+\hat{G}\widetilde{C}_1+\hat{G}\widetilde{D}_1\hat{F})R_{11}^{-1}  \\
     & G_2=R_{11}(G_1 +\hat{G}\widetilde{D}_1)\hat{U} \\
     & \widetilde{D}_{2}=\hat{V}\widetilde{D}_1\hat{U} \mbox{ and }
       \widetilde{C}_{2} = \hat{V}(\widetilde{C}_1+\widetilde{D}_1\hat{F})R_{11}^{-1}
    \end{split}
 \end{equation}
 where $\hat{F}=\begin{bmatrix} -R_{22}^{-1}R_{12} \\ 0 \end{bmatrix}$, $\hat{G}=R_{11}^{-1}H_{12}$, 
 $\hat{U}=\begin{bmatrix} R_{22}^{-1} & 0 \\ 0 & I_m \end{bmatrix}$,
 and 
 $\hat{V}=H_{22}$.

 We then claim that the following choice of matrices
\begin{equation}
\label{feedback:pf:eq4}
\begin{split}
   & T=R_{11} \mbox{ and } U=L_1^{+}\hat{U}L_2  \\
   & F=L_{1}^{+}(\hat{F}+\hat{U}F_2R_{11}-F_1) 
\end{split}
\end{equation}
satisfies \eqref{feedback:pf:eq1}.
We prove \eqref{feedback:pf:eq1:eq1} -- \eqref{feedback:pf:eq1:eq6}
one by one.

\emph{Proof of \eqref{feedback:pf:eq1:eq1}:}
 Indeed, from \eqref{feedback:pf:eq2} it then follows that 
 $\mathscr{S}_1$ and
 $\mathscr{S}_2$ are related by a feedback equivalence with output injection
 $(R_{11},\hat{F},\hat{G},\hat{U},\hat{V})$.
 From \cite[page 169, Exercise 7.1]{TrentelmanBook} it follows that
 $\V_2=R_{11}(\V_1)=T\V_1$.

\emph{Proof of \eqref{feedback:pf:eq1:eq2}:}
From the definition of $F_2$ it follows
$(\widetilde{C}_2+\widetilde{D}_2F_2)\V_2=\{0\}$, and
$(A_2+G_2F_2)\V_2 \subseteq \V_2$.
Substituting the expressions for $\widetilde{C}_2,\widetilde{D}_2,A_2,G_2$
from \eqref{feedback:pf:eq2} and using that $\V_2=R_{11}\V_1$ and
that $R_{11},\hat{V}$ are invertable, it follows that
for all $x \in \V_1$, 
\begin{equation}
\label{feedback:pf:eq1:eq21}
\begin{split}
 & (\widetilde{C}_1+\widetilde{D}_1(\hat{F}+\hat{U}F_2R_{11}))x=0 \\
 & (A_1+G_1(\hat{F}+\hat{U}F_2R_{11})+\hat{G}\widetilde{C}_1+\hat{G}\widetilde{D}_1(\hat{F}+\hat{U}F_{2}R_{11}))x=\\
& (A_1+G_1(\hat{F}+\hat{U}F_2R_{11}))x \in V_1 \\
\end{split}
\end{equation}
Hence, $(A_1+G_1F_1+G_1F)\V_1 \subseteq \V_1$.

\emph{Proof of \eqref{feedback:pf:eq1:eq3}:}
Since from the definition of $F_1$ it follows that
$(A_1+G_1F_1)x \in V_1$,
$(\widetilde{C}_1x+\widetilde{D}_1F_1x)=0$,
 for all $x \in \V_1$,
from \eqref{feedback:pf:eq1:eq21}, 
 it then follows that
for all $x \in \V_1$, $G_1(\hat{F}+\hat{U}F_2R_{11}-F_1)x \in \V_1$ and
$\widetilde{D}_1(\hat{F}x+\hat{U}F_2R_{11}-F_1)x=0$.
Hence, $(\hat{F}+\hat{U}F_2R_{11}-F_1)x \in \IM L_1$ and hence
\begin{equation*}
\label{feedback:pf:eq1:eq31}
L_1Fx=L_1L_1^{+}(\hat{F}+\hat{U}F_2R_{11}-F_1)x=(\hat{F}+\hat{U}F_2R_{11}-F_1)x
\end{equation*}
for all $x \in \V_1$. From this it follows that
\begin{equation}
\label{feedback:pf:eq1:eq32}
 x \in \V_1:
    F_1x+L_1Fx=(\hat{F}+\hat{U}F_2R_{11})x.
\end{equation}
From \eqref{feedback:pf:eq1:eq32} it then follows that 
$(A_1+G_1F_1+G_1L_1F)x=A_1x+G_1(\hat{F}+\hat{U}F_2R_{11})x$ for all
$x \in \V_1$. From this and \eqref{feedback:pf:eq2},
\eqref{feedback:pf:eq1:eq3} follows.

\emph{Proof of \eqref{feedback:pf:eq1:eq4}:}
Recall that 
 \( \IM L_2=\ker (\hat{V}\widetilde{D}_1\hat{U}) \cap (R_{11}G_1\hat{U})^{-1}(\V_2)=\hat{U}^{-1}(\ker \widetilde{D}_1 \cap G_1(\V_1))=\hat{U}^{-1}\IM L_1$.
Since $\hat{U}$ is invertable, it follows that $\Rank L_1=\Rank L_2=k$
and that 
\begin{equation}
\label{feedback:pf:eq1:eq41}
L_1U=L_1L_1^{+}\hat{U}L_2=\hat{U}L_2.
\end{equation}
Hence, using 
\eqref{feedback:pf:eq2} and $\widetilde{D}_2\hat{U}L_2=0$, it follows
that 
$TG_1L_1U=TG_1\hat{U}L_2=TG_1\hat{U}L_2+T\hat{G}\widetilde{D}_2\hat{U}L_2=G_2L_2$.

\emph{Proof of \eqref{feedback:pf:eq1:eq5}:}
 It is eaasy to see that 
\eqref{feedback:pf:eq1:eq5} is equivalent to
\begin{equation}
\label{feedback:pf:eq1:eq51}
 \begin{bmatrix} T_2 & 0 \\ 0 & I_m \end{bmatrix}^{-1} \begin{bmatrix} T_1 & 0 \\ 0 & I_m \end{bmatrix} 
  \begin{bmatrix} 0_{r \times k} \\ L_1U \end{bmatrix} = \begin{bmatrix} 0_{r \times k}  \\ L_2 \end{bmatrix}.
\end{equation}
We will show 
\eqref{feedback:pf:eq1:eq51}
To this end, notice that
\begin{equation}
\label{feedback:pf:eq1:eq52}
\begin{split}
& \begin{bmatrix} T_2 & 0 \\ 0 & I_m \end{bmatrix}^{-1} \begin{bmatrix} T_1 & 0 \\ 0 & I_m \end{bmatrix} = \begin{bmatrix}
  T_2^{-1}T_1 & 0 \\ 0 & I_m \end{bmatrix}= \\
  & \begin{bmatrix}
   R_{11} & 0 & 0 \\
   R_{21} & R_{22} & 0 \\
   0      & 0      & I_m
  \end{bmatrix}= 
  \begin{bmatrix} R_{11} & 0 \\ \begin{bmatrix} R_{21} \\ 0 \end{bmatrix} & \hat{U}^{-1} \end{bmatrix}
\end{split}
\end{equation}
Hence,
% with the decomposition 
%$L_1=\begin{bmatrix} L_{11}^T & L_{12}^T \end{bmatrix}$ where
%$L_11 \in \mathbb{R}^{(n-r) \times k}$ and 
% $L_12 \in \mathbb{R}^{m \times k}$.
\begin{equation}
\label{feedback:pf:eq1:eq53}
  \begin{bmatrix} T_2 & 0 \\ 0 & I_m \end{bmatrix}^{-1} \begin{bmatrix} T_1      & 0 \\ 0 & I_m \end{bmatrix}  
  \begin{bmatrix} 0_{r \times k} \\ L_1U \end{bmatrix} =
  \begin{bmatrix} 0_{r \times k} \\
   \hat{U}^{-1}L_1U
 \end{bmatrix}.
\end{equation}
Using $L_1U=\hat{U}L_2$ proven above in \eqref{feedback:pf:eq1:eq41}, 
it follows that $\hat{U}^{-1}L_1U=L_2$ and hence 
\eqref{feedback:pf:eq1:eq53} implies 
\eqref{feedback:pf:eq1:eq51}.

\emph{Proof of \eqref{feedback:pf:eq1:eq6}:}
 Again, it is enough to show that
 \begin{equation}
 \label{feedback:pf:eq1:eq61}
 \begin{split}
    & \forall x \in \V_1: \\ 
    & \begin{bmatrix} T_2 & 0 \\ 0 & I_m \end{bmatrix}^{-1}
    \begin{bmatrix} T_1 & 0 \\ 0 & I_m \end{bmatrix} \begin{bmatrix} I_r \\ (F_1+L_1F) \end{bmatrix} x= 
\begin{bmatrix} I_r \\ F_2 \end{bmatrix} R_{11} x.
\end{split}
\end{equation}
 From \eqref{feedback:pf:eq1:eq52} it follows that
 \begin{equation}
 \label{feedback:pf:eq1:eq62}
 \begin{split}
 & \begin{bmatrix} T_2 & 0 \\ 0 & I_m \end{bmatrix}^{-1}
   \begin{bmatrix} T_1 & 0 \\ 0 & I_m \end{bmatrix} \begin{bmatrix} I_r \\      (F_1+L_1F) \end{bmatrix} = \\
 &  \begin{bmatrix} R_{11} \\ \begin{bmatrix} R_{21} \\ 0 \end{bmatrix} + \hat{U}^{-1}(F_1+L_1F) \end{bmatrix}.
 \end{split}
 \end{equation}
Notice that 
$\hat{U}^{-1}\hat{F}=\begin{bmatrix} -R_{21}^T \\ 0 \end{bmatrix}$ and
hence, using \eqref{feedback:pf:eq1:eq32}, 
\[ 
  \begin{split}
  & \begin{bmatrix} R_{21} \\ 0 \end{bmatrix}x + \hat{U}^{-1}(F_1+L_1F)x= \\
  & \begin{bmatrix} R_{21} \\ 0 \end{bmatrix}x + 
  \hat{F}x+F_2R_{11}x = F_2R_{11}x
 \end{split}
\]
for all $x \in \V_1$.
Combining this with \eqref{feedback:pf:eq1:eq62},
\eqref{feedback:pf:eq1:eq61} follows easily.

\end{proof}
%%\begin{proof}[Proof of Lemma \ref{opt:contr:unique}]
%% From Lemma \ref{} it follows that there exists a
%% state-input isomorphism $(T,U)$ from
%% $(A_1,B_1,C_1,D_1)$ to $(A_2,B_2,C_2,D_2)$ which, moreover, has
%% the property that $T\LMAPINV_1=\LMAPINV_2$.
%% It is easy to see that state-input isomorphisms preserves stabilizability,
%% hence $(A_2,B_2)$ is stabilizable.
%% It is left to show that  if
%% \( hat{Q}_{0,1} A_1+A_1^T\hat{Q}_{0,1} +\hat{Q}_{0,1}B_1B_1^T \hat{Q}_{0,1} + C_1^TC_ \le 0\), then
%% \( hat{Q}_{0,2} A_2+A_2^T\hat{Q}_{0,2} +\hat{Q}_{0,2}B_2B_2^T \hat{Q}_{0,2} + C_2^TC_2 \le 0\),
%% where $Q_{0,i}=({\LMAPINV}^{+}_i)^T(E^{+}E-\LOPT)^TQ_0(E^{+}E-\LOPT){\LMAPINV}^{+}_i$, $i=1,2$.
%%It is easy to see that
%%$T^{T}Q_{0,2}T=Q_{0,1}$. Hence, by multiplying \( hat{Q}_{0,1} A_1+A_1^T\hat{Q}_{0,1} +\hat{Q}_{0,1}B_1B_1^T \hat{Q}_{0,1} + C_1^TC_1 le 0\) by $T^{T}$ from the left and
%% $T$ from the right, it follows that
%%\( hat{Q}_{0,2} A_2+A_2^T\hat{Q}_{0,2} +\hat{Q}_{0,2}B_2B_2^T \hat{Q}_{0,2} + C_2^TC_2         \le 0\) holds.
%%\end{proof}

\end{document}